\documentclass[reqno]{amsart}
\newcommand \datum {January 17, 2024}

\usepackage{amssymb,latexsym}
\usepackage{amsmath}
\usepackage{hyperref}
\usepackage{url}
\usepackage{graphicx}
\usepackage[dvipsnames]{xcolor}
\usepackage{enumerate}
\numberwithin{equation}{section}
\theoremstyle{plain}
 \newtheorem{theorem}{Theorem}[section]
 \newtheorem{lemma}[theorem]{Lemma}
 \newtheorem{observation}[theorem]{Observation} 

\theoremstyle{definition}
 \newtheorem{fact}[theorem]{Fact} 
 \newtheorem{definition}[theorem]{Definition}
 \newtheorem{example}[theorem]{Example}
 \newtheorem{remark}[theorem]{Remark}

\newcommand \idl[1] {\textup{idl}(#1)}
\newcommand \prline [2] {\ell_{#1,#2}}
\newcommand \vvec[1] {\vec #1\kern 0.75pt} 
\newcommand \pvec [1]{\vvec #1'}
\newcommand \hij [2] {h_{#1,#2}}
\newcommand \mng [1] {f^\textup{mng}(#1)}
\newcommand \nothing[1]{}
\newcommand \fer {f^{(\textup e)}}
\newcommand \finu [2] {f_{#1,#2}}
\newcommand \fif {f_{i,4}}
\newcommand \ginu [2] {g_{#1,#2}}
\newcommand \gif {g_{i,4}}
\newcommand \ter[1] {#1^\ast} 
\newcommand \ater [2] {#2^{\ast #1}}
\newcommand \bter [2] {#2^{\ast\ast #1}} 
\newcommand \nater [1] {#1^{\ast}}
\newcommand \uid[2] {\vvec #1^{(#2)}}
\newcommand \iud[2] {\vvec #2^{(#1)}}
\newcommand \nuid[2] {#1^{(#2)}}
\newcommand \inud[2] {#2^{{#1}}}
\newcommand \biud[3] {\vvec #3^{(#1,#2)}}
\newcommand \vvxi {\vec \xi\,)}
\newcommand \vspan [1] {[#1]_{\textup{vs}}}
\newcommand \vgener [1] {\textup{Span}(#1)}
\newcommand \fvgener [2] {\textup{Span}_{#1}(#2)}
\newcommand \latgen [1] {[#1]_{\textup{lat}}}
\newcommand \kdelta[3] {\delta^{(#1)}_{#2#3}} 
\newcommand \vxi {{\vec \xi}}
\newcommand \pxi {\vec \xi^+}
\newcommand \pjg {u_j,\vvec g}
\newcommand \Subp[2] {\textup{Sub}(_{#1} #2) } 
\newcommand \Sub[1] {\textup{Sub}(#1) } 
\newcommand \typ[1]  {\textup{typ}(#1)}
\newcommand \hgh  {\textup{hgh}}
\newcommand \ftyp[1]  {\textup{ftyp}(#1)}
\newcommand \coring[2] {R\langle #1,#2\rangle}
\newcommand \ppr {\kern1pt'}
\newcommand\bpf[4] {F\bigl( \textstyle{{{#1}\atop {#3}} {{#2}\atop {#4}}\bigr) }}
\newcommand\bplus[3]{\mathrel{ \oplus_{#1#2#3} } }
\newcommand\bszor[3]{\mathrel{ \otimes_{#1#2#3} } }
\newcommand\bminus[3]{\mathrel{ \ominus_{#1#2#3} } }
\newcommand \recip[3] {\textup{rec}_{#1#2#3}}
\newcommand \ipsp[1] {P^{#1}_{2}}
\newcommand \psp[1] {P_{#1}}
\newcommand \psd {P_{d-1}}
\newcommand \vecspace[2]{{_#1  #1^#2 }} 
\newcommand \bmu {\mu}
\newcommand \ul[1] {\underline{#1}}
\newcommand\lint[1]{\lfloor #1\rfloor}
\newcommand\uint[1]{\lceil #1\rceil}
\renewcommand \phi{\varphi}
\newcommand \Nnul {{\mathbb N}_0}
\newcommand \Nplu {{\mathbb N^+}}
\newcommand{\tbf}{\textbf}
\newcommand{\set}[1]{\{#1\}}
\newcommand \red[1]{{\textcolor{red}{#1}\color{black}}}

\begin{document}

\title[Generating  subspace lattices]
{Generating subspace lattices, their direct products, and their direct powers}

\author[G.\ Cz\'edli]{G\'abor Cz\'edli}
\email{czedli@math.u-szeged.hu}
\urladdr{http://www.math.u-szeged.hu/~czedli/}
\address{University of Szeged, Bolyai Institute. 
Szeged, Aradi v\'ertan\'uk tere 1, HUNGARY 6720
}

\begin{abstract} 
In 2008, L\'aszl\'o Z\'adori proved that the lattice $\Sub V$ of all subspaces of a vector space $V$ of finite dimension at least $
3$  over a finite field $F$ has a 5-element generating set; in other words, $\Sub V$ is 5-generated. 
We prove that the same holds over every $1$- or $2$-generated field; in particular, over every field that is a finite degree extension of its prime field. 
Furthermore, let 
$F$, $t$, $V$, $d\geq 3$,  $\lint{d/2}$, and $m$ denote an arbitrary field, the minimum cardinality  of a generating set of $F$, a finite dimensional vector space over $F$, the dimension (assumed to be at least $3$) of $V$,    the integer part of $d/2$, and 
the least cardinal such that $m\lint{d^2/4}$ is at least $t$, respectively. We prove that $\Sub V$ is $(4+m)$-generated but none of its generating sets is of size less than $m$. Moreover, the $k$-th direct power of $\Sub V$ is $(5+m)$-generated for many positive integers $k$; for all positive integers $k$ if $F$ is infinite. Finally, let $n$ be a positive integer. For  $i=1,\dots, n$, let  $p_i$ be a prime number or 0, and let $V_i$ be the 3-dimensional vector space over the prime field of characteristic $p_i$. We prove that the direct product of the lattices $\Sub{V_1}$, \dots, $\Sub{V_n}$  is 4-generated if and only if each of the numbers $p_1$, \dots, $p_n$ occurs at most four times in the sequence  $p_1$, \dots, $p_n$. Neither this direct product nor any of the subspace lattices $\Sub V$ above is 3-generated. 
\end{abstract}

\dedicatory{Dedicated to Honorary Professor J\'ozsef N\'emeth on his eightieth birthday}

\thanks{This research was supported by the National Research, Development and Innovation Fund of Hungary, under funding scheme K-138892.  \hfill{\red{\tbf{\datum}}}}




\maketitle

\section{Note on the dedication}

At the beginning of my university studies, Dr.\ J\'ozsef N\'emeth  taught me in the first semester. He was excellent. All the students in the classroom regretted that he was assigned different sections and  courses for the next semester.   
As I reminisce about his unsurpassable tutorials, I wish him a happy birthday.

\section{Introduction}

For a lattice or a field $A$, we define the following cardinal number: 
\begin{equation}
\mng A:=\min\set{|X|: X\text{ is a generating set of }A}.
\label{eq:rgDtWvStsgHtkw}
\end{equation}
For later reference, note that for a field $F$,
\begin{equation}
F\text{ is a prime field if and only if }\mng F=0.
\label{eq:crMdtjlcnmpGSn}
\end{equation}
By a \emph{field} we mean a \emph{commutative field}.
Let $L$ be the subspace lattice of a vector space $V$ 
of finite dimension   
$d\geq 3$ over a field $F$; in notation, 
\begin{equation}
L:=\Subp F V\text{, where }V=\vecspace F d\text{ is }3\leq d\text{-dimensional}.
\label{eq:RsPlWrQn}
\end{equation}
We often write $\Sub V$ instead of $\Subp F V$. 
 Z\'adori \cite{zadori2} proved that whenever $F$ is a finite field, then $L$ in \eqref{eq:RsPlWrQn} is 5-generated. Earlier, Gelfand and Ponomarev \cite{gelfand} proved that $L$ is 4-generated but not 3-generated if $F$ is a prime field; see Z\'adori \cite{zadori2} for historical details.  

Our aim is to generalize these two results and prove some related results. 
In Z\'adori's result, $F$ is a finite field with $\mng F=1$; we remove finiteness from his assumptions on $F$ and, instead of $\mng F=1$, we assume only that $\mng F\in\set{1,2}$. 
Related to Gelfand and Ponomarev' result,  we  prove that if $d=3$ and $F$ is a prime field, then 
$\mng{L^k}=4$ holds for $L$ from \eqref{eq:RsPlWrQn} even for $k\in\set{2,3,4}$ (in addition to $k=1$); the number 4 is optimal here at both of its occurrences. Furthermore, we extend this result to direct products; so the just-mentioned result (for $k$-th direct powers, $k\in\set{1,2,3,4}$) becomes a particular case.
If no peculiarity of the (finite or infinite) cardinal number $\mng F$ is assumed and $L$ is still from \eqref{eq:RsPlWrQn}, then denote by $m$ the smallest cardinal number such that $m\lint{d^2/4} \geq \mng F$. We prove that $m\leq \mng L\leq 4+m$ and $\mng{L^k}\leq 5+m$ for many integers $k\in\Nplu:=\set{1,2,3,\dots}$; for all $k\in\Nplu$ if $F$ is infinite.

By a \emph{nontrivial lattice} we mean an at least 2-element lattice.
In Section \ref{sect:aproof}, to shed more light on $\mng{L^k}$, we prove the following observation, in which $L$ need not be a subspace lattice.

\begin{observation}\label{obs:hbNvlnlgnD}
Let $L$ be a  nontrivial lattice and let $n\in\Nplu:=\{1$, $2$, $3,\dots\}$.  If $k\in\Nplu$ is large enough to exclude the existence of a $k$-element antichain in $L^n$, then $L^k$ is not $n$-generated. In particular, $L^k$ is not $n$-generated if $k>|L|^n$.
\end{observation}

Finally, note that in addition to earlier results on the generation of subspace lattices, a possible connection with cryptology also  motivates the study of small generating sets of lattices; see Cz\'edli \cite{czgboolegen}.

\subsection*{Outline} Section \ref{sect:results} formulates  exactly the results mentioned so far in three theorems, and presents some related statements. Section \ref{sect:coord} recalls some well-known basic facts from coordinatization theory. 
Each of Sections \ref{sect:aproof},  \ref{sect:bproof}, and  \ref{sect:cproof} proves one of the three theorems together with some auxiliary statements. Section \ref{sect:append-ZLprfn} points out how one can extract Gelfand and Ponomarev's result,  quoted right after \eqref{eq:RsPlWrQn}, from Z\'adori's proof given in \cite{zadori2}. Section \ref{sect:maple} presents two Maple programs related to Sections \ref{sect:results} and \ref{sect:bproof}.

\section{The main results and some of their corollaries}
\label{sect:results}

Recall that for $0\leq r\leq m\in\Nplu$ and a prime power $q$, the \emph{Gaussian binomial coefficient} is defined as
\begin{equation}
\binom m r_q:=\frac{(1-q^{m})(1-q^{m-1})\cdots(1-q^{m-r+1})}{(1-q)(1-q^{2})\cdots(1-q^{r)}};
\label{eq:GssnbncFfcnt}
\end{equation}
see, e.g., O'Hara \cite{ohara}\footnote{\href{https://en.wikipedia.org/wiki/Gaussian_binomial_coefficient}{https://en.wikipedia.org/wiki/Gaussian\textunderscore{}binomial\textunderscore{}coefficient} would also do.}.
For convenience, let us agree that for a cardinal $\lambda$, 
\begin{equation}
\text{if }1\leq r\leq m-1\in\Nplu \text{ and }\lambda\geq \aleph_0\text{, then we let }
\binom m r_{\lambda}:=\lambda.
\label{eq:qbinomalephNl}
\end{equation}
This convention is motivated by the fact that \eqref{eq:GssnbncFfcnt}
is known to be the number of the $r$-dimensional subspaces of the\footnote{As the definite article indicates, the $m$-dimensional vector space over a given field in the paper is understood up to isomorphism but its subspaces are not.} $m$-dimensional vector space over the $q$-element field; now the same holds for every $\lambda$-element field in virtue of \eqref{eq:qbinomalephNl}. The upper integer part and the lower integer part of a real number $x$ will be denoted by $\uint x$ and $\lint x$, respectively; for example, $\uint{\sqrt {80}}=\uint 9 =9$ and $\lint{\sqrt {80}}=\lint 8 =8$. More generally, let us agree that for a cardinal number $t$ and a positive integer $n$,
\begin{equation}
\uint{t/n}:=\min\set{m: m n\geq t}; \text { it is a cardinal number.}
\label{eq:flSgrzsmSslsBgr}
\end{equation}

\begin{theorem}\label{thmnok} As in \eqref{eq:RsPlWrQn},
assume that $L=\Subp F V$, where $F$ is an arbitrary field, $3\leq d\in\Nplu$, and $V$ is the $d$-dimensional vector space over $F$. Let
$t:=\mng F$, the minimum of the cardinalities of the generating sets of $F$; see \eqref{eq:rgDtWvStsgHtkw}. Then 
\begin{equation}
4\leq \mng L\leq 4+ \bigl\lceil \frac{t}{\lint{d^2/4}}\bigr\rceil.
\label{eq:nKmZfRtdrCh}
\end{equation} 
\end{theorem}

For $t=0$ or $t\in\set{1,2}$,  \eqref{eq:nKmZfRtdrCh} implies that $\mng L = 4$ or $\mng L\leq 5$, respectively. Thus, the results quoted from Z\'adori \cite{zadori2} and Gelfand and Ponomarev \cite{gelfand}  after \eqref{eq:RsPlWrQn} are particular cases of Theorem \ref{thmnok}. Note that 
\begin{equation}
\lint{d^2/4}=\lint{d/2}\cdot\uint{d/2}\text{ and so }
\bigl\lceil \frac{t}{\lint{d^2/4}}\bigr\rceil
=\bigl\lceil \frac{t}{\lint{d/2}\cdot\uint{d/2}}\bigr\rceil
\label{eq:nNzImdrHlTrmLj}
\end{equation}
hold for $3\leq d\in\Nplu$. 
If $t$ is large compared to $d$, then \eqref{expl:mZvsDnStSg} below gives a better lower bound for $\mng L$ than the first inequality in \eqref{eq:nKmZfRtdrCh}.

\begin{theorem}\label{thmk} Let $F$ be a field, let $3\leq d\in\Nplu$, and denote by $V$ and $L$ the $d$-dimensional vector space over $F$ and its subspace lattice $\Subp F V$, respectively. Let $k\in\Nplu$ and, with reference to  \eqref{eq:GssnbncFfcnt} and \eqref{eq:qbinomalephNl},  let 
\begin{equation}
\bmu:=\binom d{\lint{d/2}}_{|F|}.
\label{eq:ckHlrgjCnpzgN}
\end{equation}
Then, using the notations of  \eqref{eq:rgDtWvStsgHtkw}, \eqref{eq:crMdtjlcnmpGSn}, \eqref{eq:flSgrzsmSslsBgr}, and \eqref{eq:nKmZfRtdrCh} and letting $t:=\mng F$,  the following
inequalities and equalities hold for $\mng{L}$ and $\mng{L^k}$:
\begin{gather}
\bigl\lceil \frac{t}{\lint{d^2/4}}\bigr\rceil  \leq \mng {L} \leq \mng {L^k}, 
\label{eq:fDnKnCspJha}\\
\text{if  } k \leq \bmu,\ \text{ then }\ \mng {L^k} \leq  5+ \bigl\lceil \frac{t}{\lint{d^2/4}}\bigr\rceil  ,
\label{eq:fDnKnCspJhb}\\
\mng{L^k} = 4 \text{ provided that }t=0,\ d=3\text{, and } k\in\set{1,2,3,4},\text{ and}
\label{eq:ktTlTkvrglLk}\\
\mng{L^k} = 5 \text{ provided that }t=0,\ d=3,\ k\in\Nplu\text{, and } 5\leq k\leq\bmu.
\label{eq:mtPhfGnprWmt}
\end{gather} 
\end{theorem}

As $\mu$ can be an infinite cardinal number, \eqref{eq:mtPhfGnprWmt} repeats that $k\in\Nplu$.

\begin{theorem}\label{thm:prd} Let $\lambda$ be a nonzero ordinal number, and assume that for each $\iota<\lambda$, $V_\iota$ is the $3$-dimensional vector space over a prime field $F_\iota$. Let $L$ be the direct product of the corresponding subspace lattices, that is, 
\begin{equation}
L:=\prod_{\iota<\lambda} \Sub{V_\iota}.
\label{eq:tSvnRnkDglhlg}
\end{equation}
Then $\mng L=4$ if and only if $\lambda$ is finite, $\lambda\neq 0$, and, up to isomorphism,  each prime field occurs at most four times in the sequence $(F_\iota:\iota<\lambda)$.  
\end{theorem}

It does not seem to be easy to generalize \eqref{eq:ktTlTkvrglLk} and \eqref{eq:mtPhfGnprWmt} to $3<d\in\Nplu$.  
Table \ref{table:szbklngy}, obtained by computer algebra\footnote{Maple V, see Footnote \ref{foOt:mapLe} for more details, but many others would also do.}, shows that the Gaussian binomial coefficient $\bmu$ occurring in \eqref{eq:ckHlrgjCnpzgN} is large in general. 
\begin{table}[htb]
\begin{tabular}{l|r|r|r|r|r|r|r|r|r|r}
$q=$ &  2& 3& 4& 5  \cr
\hline
$\bmu\approx$ & $1.540\cdot 10^{482}$ & 
$  4.423\cdot 10^{763}$& $ 2.871 \cdot 10^{963}$ & $ 2.958 \cdot 10^{1118}$\cr
\hline\hline
$q=$ &  7& 8& 9& 11  \cr
\hline
$\bmu\approx$ & $ 1.715 \cdot 10^{1352}$ & 
$ 1.023\cdot 10^{1445}$& $  7.002 \cdot 10^{1526}$ & $1.878 \cdot 10^{1666}$\cr
\hline\hline
$q=$ &  13& 16& 17& 19  \cr
\hline
$\bmu\approx$ & $ 2.223\cdot 10^{1782}$ & 
$  4.186 \cdot 10^{1926}$& $ 5.574  \cdot 10^{1968}$ & $ 1.073 \cdot 10^{2046}$\cr
\hline\hline
\end{tabular}
\caption{For $d=80$, the approximate values of some Gaussian binomial coefficients occurring in \eqref{eq:ckHlrgjCnpzgN} 
} 
\label{table:szbklngy}
\end{table}

The following remark is trivial since $L^h$ and $\prod_{i\in S}L_i$ in it are  homomorphic images of $L^k$ and $\prod_{i\in [k]} L_i$ (where $[k]=\set{1,\dots,k}$),  respectively.

\begin{remark}\label{rem:sbrStvgnBx}
For a lattice $L$ and $h,k,n\in\Nplu$ such that $h<k$, if $L^k$ is $n$-generated, then $\mng{L^h}\leq n$. More generally, if  $\prod_{i\in [k]} L_i$  is $n$-generated  and $S\subseteq [k]$,  then $\prod_{i\in S}L_i$ has an at most $n$-element generating set. 
\end{remark}

The following easy lemma could be of separate interest. 
For a subset $X$ of a vector space $V$ over a field $K$, 
let $\fvgener K X$ denote the subspace of $V$ generated by $X$; we can also write $\vgener X$ if $K$ is clear from the context.

\begin{lemma}\label{lemma:bgzsLm}
Let $F$ be a field with a subfield $P$ (that is, let $F\vert P$ be a field extension) and let $3\leq d\in \Nplu$.  Furthermore, let 
$V'= \vecspace P d$ and $V= \vecspace F d$ be the $d$-dimensional vector spaces $($consisting of $d$-tuples$)$ over $P$ and $F$, respectively. 
Then 
\begin{equation}
\phi\colon \Subp{P}{V'}\to \Subp F V,\text{ defined by }X\mapsto \fvgener F X,
\label{eq:mMtlnZknZltcwrTb}
\end{equation}
is a lattice embedding. Furthermore, $\phi$ preserves the length, the covering relation, the smallest element $0$, and the largest element $1$. We also have that 
for any subset $H$ of $V'$, $\phi(\fvgener P H)=\fvgener F H$.
\end{lemma}

In the forthcoming Example \ref{expl:mZvsDnStSg}, to be proved in  Section \ref{sect:cproof}, the number $80$ makes one and a half dozen appearances. Although most instances could be replaced by any positive integer greater than $1$, we have opted for $80$ in keeping with the paper's dedication.

\begin{example}\label{expl:mZvsDnStSg} Let $F$ be a field and let $3\leq d\in\Nplu$. Let $L$ stand for the subspace lattice $\Sub V=\Sub{\vecspace F d}$ of the $d$-dimensional vector space $V$ over $F$. Then  the following six assertions hold.

\textup{(a)} If $\alpha_1$, \dots, $\alpha_{80}$ are (not necessarily distinct) algebraic irrational numbers over the field $\mathbb Q$ of rational numbers and  $F=\mathbb Q(\alpha_1,\dots,\alpha_{80})$ is the field that these numbers generate,  then  $L$ 
has a $5$-element generating set. Furthermore, for every $2\leq k\in\Nplu$,  $L^k$ has a $6$-element generating set. In particular, if 
\[F=\mathbb Q\bigl(\sqrt{2023},\sqrt 2,  \sqrt[\leftroot{2}\uproot{2}{3}]{3},   \sqrt[\leftroot{2}\uproot{2}{4}]{4},
\sqrt[\leftroot{2}\uproot{2}{5}]{5},
\sqrt[\leftroot{2}\uproot{2}{6}]{6},\dots,
\sqrt[\leftroot{2}\uproot{2}{80}]{80}\,
\bigr),
\]
then $L^{80}$ has a $6$-element generating set. 

\textup{(b)} Let $\beta_1$, \dots, $\beta_{80}$ be algebraically independent transcendental numbers over $\mathbb Q$ and let  $F:=\mathbb Q(\beta_1,\dots,\beta_{80})$. If $d=3$, then $L$ has a $44$-element generating set and each of its generating sets consists of at least $40$ elements. If $d=8$, then $L$ has a $9$ element generating set but not a $4$-element one.

\textup{(c)} If $\gamma_1$, \dots, $\gamma_{80}$ are algebraically independent transcendental numbers over $\mathbb Q$,  $F=\mathbb Q(\gamma_1,\dots,\gamma_{80})$,  $d=80^{80}$, and $k=80^{80d}$, then $L$ has a $5$-element generating set and $L^k$ has a $6$-element one.

\textup{(d)} If $|F|=19$ or $F=\mathbb Q$, $d=80$, and $k=10^{2046}$, then $L^k$  can be generated by five elements.

\textup{(e)} If $F=\mathbb A$, the field of algebraic numbers, then $L$ is not finitely generated.

\textup{(f)} If $F=\mathbb Q(\pi^{80}, \sqrt[80]{80})$, where $\pi\approx 3.141\,592\,653\,589\,793$ is the well-known transcendental constant, 
then $L$ has a $5$-element generating set while $L^{80}$ has a $6$-element one.  
\end{example}

\begin{remark}\label{rem:ncdGnnJhNkszN}
For  $F=\mathbb Q(\pi^{80}, \sqrt[80]{80})$ in Example \ref{expl:mZvsDnStSg}\textup{(f)}, $\mng F=2$.
\end{remark}

\nothing{
\begin{remark}
From a result announced in Herrmann, Ringel, and Wille \cite{herrmannringelwille},
Z\'adori \cite{zadori2} derived that $\Subp F V$ cannot be generated by four elements provided that $F$ is a finite non-prime field and $3\leq \dim (V)$ is finite. Although his argument seems to work with the assumption $\mng F>0$ without finiteness and seems to provide a better lower bound for $\mng {L}$  than \eqref{eq:fDnKnCspJha} for  $d$  large, we do not follow this plan in our theorems.
\end{remark}}

\section{Some basic facts from the coordinatization theory of lattices}\label{sect:coord}

The proof of Theorem \ref{thmnok} grew out from  the coordinatization theory of Arguesian lattices. This theory was introduced by  J.\ von Neumann; see, for example, Artmann \cite{artmann}, Day and Pickering \cite{daypick}, Freese \cite{freese}, Herrmann \cite{herrmann} and \cite{herrmann84}, 
and von Neumann \cite{neumann,bookvonneumann}. As these papers but Herrmann \cite{herrmann} and \cite{neumann} are referenced in Cz\'edli and Skublics \cite{czgskublics}, where the treatment and the notations are unified, it will be convenient to reference also \cite{czgskublics}\footnote{At the time of writing, a preprint of this paper is freely available from 
\href{http://tinyurl.com/czedli-skublics}{http://tinyurl.com/czedli-skublics} or, equivalently, it can be found in the author's website, 
\href{https://www.math.u-szeged.hu/~czedli/}{https://www.math.u-szeged.hu/~czedli/} = 
\href{http://tinyurl.com/g-czedli}{http://tinyurl.com/g-czedli} .} even though no result that was first proved in \cite{czgskublics} is needed here. 
Actually, we need only the easy first step from coordinatization theory, and the statements of this section are straightforward to verify with elementary computations in Linear Algebra. 
In the paper, we  often use the notation
\begin{equation*}
[i]:=\set{1,2,\dots,i}\text{ for }i\in\Nnul;\text{ in particular, }[0]:=\emptyset.
\end{equation*}
As a general assumption for the whole section, we assume that $F$ is a field, $3\leq d\in\Nplu$, and   $V=\vecspace F d$ is the $d$-dimensional vector space over $F$.  We let $v_i:=(0,\dots,0,1,0,\dots,0)\in V$, with $1$ at the $i$-th position, for $i\in [d]$. We turn $V=\vecspace F d$ into the $(d-1)$-dimensional projective space $\psd=\psd(F)$ over $F $ in the usual way except that we use $-1$ instead of $1$ for ``finite'' points\footnote{The $-1$ is explained by the minus sign in von Neumann's choice of $c_{i,j}=F (v_i-v_j)$, see later, and by our intention that the unit $c_{1,4}$ of $\coring 4 1$, to be defined soon,  in Figure \ref{figinv} should be to the right of the  zero  $a_4$ of the ring.}; see, e.g., Figure \ref{figpsd}.

\begin{figure}[ht] 
\centerline{ \includegraphics[scale=1.0]{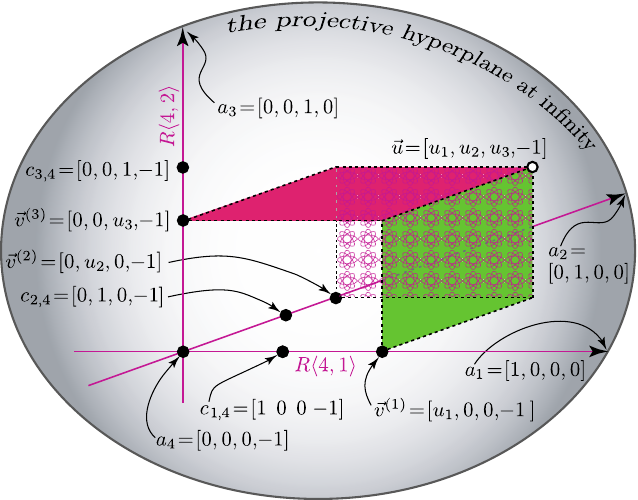}} \caption{The 3-dimensional projective space}\label{figpsd}
\end{figure}

The points and the lines of  $\psd$ are the 1-dimensional subspaces and the 2-dimensional subspaces of $V$, respectively. A 1-dimensional subspace of $V$ is either of the form $F (x_1,\dots,x_{d-1},-1)$ and then  $[x_1,\dots,x_{d-1},-1]$ denotes 
(in other words, coordinatizes)
the corresponding (\emph{projective}) \emph{point} of $\psd$, or this subspace is of the form $F (x_1,\dots,x_{d-1},0)$ and then 
 $[x_1,\dots,x_{d-1},0]$ stands for the corresponding projective point. We call the points of the form $[x_1,\dots,x_{d-1},0]$ \emph{points at infinity}  (even if $F $ is finite and thus so is $\psd$); the rest of the points are said to by \emph{finite points}. 
The finite points form the $(d-1)$-dimensional affine space over $F $. As usual,
this affine space visualizes $\psd$ so that the finite points are the points of the affine space, while an infinite projective point $[x_1,\dots,x_{d-1},0]$ is the \emph{direction} $(x_1,\dots,x_{d-1})$ in the affine space.
(Of course, $(\lambda x_1,\dots,\lambda x_{d-1})$ is the same direction and $[\lambda x_1,\dots,\lambda x_{d-1},0]$ is the same projective point at infinity for any $\lambda\in F \setminus\set 0$.) Some sort of visualization of $\psd$ for $d=4$ is given in Figure \ref{figpsd}; most parts of this figure will be used only later.

We often consider the projective space $\psd$ and a line $h$ of $\psd$ as the set of all points of $\psd$ and the set of points lying on $h$. For points $x\neq y$ in $\psd$, let $\prline x y$ denote the unique line through $x$ and $y$. Following, say, Gr\"atzer \cite[page 376]{ggfound}, a subset $X$ of $\psd$ is said to be a \emph{subspace} of $\psd$ if whenever $x$ and $y$ are distinct points in $X$, then $X$ contains all points of the line $\prline x y$. The subspaces of $\psd$ form a lattice, which we denote by $\Sub\psd=\bigl(\Sub\psd;\subseteq\bigr)$. 
For convenience (and following the traditions), if $x$ and $y$ are distinct points of $\psd$, then we often write $x\vee y$ instead of $\prline x y$, and we usually write $x\in\Sub\psd$ instead of the more precise  $\set {x} \in \Sub\psd$. When we think of their coordinates,  we denote the points of $\psd$ by $\vec x$, $\vec u$, etc..
There is a well-known isomorphism $\eta$ from $L=\Subp F V$ to the subspace lattice $\Sub\psd$. 
Namely, $\eta\colon L\to \Sub\psd$ is defined by the rule 
\begin{gather}
\eta(X):=\{P\in \psd:\text{the point }P\text{ corresponds to a}
\cr
1\text{-dimensional subspace of }X 
\} \in \Sub\psd
\label{eq:mtzTbnkStmlDsl}
\end{gather}
for $X\in \Subp F V$.
We do not make a sharp distinction between $X$ and $\eta(X)$. We use $\eta(X)$ and the projective space to explain and visualize the proofs. The respective (and straightforward) computations can be done with $X$ in $\Subp F V$ or with $\eta(X)$ in $\psd=\psd(F)$ based on the following fact, which is well known and it can easily be derived from \eqref{eq:mtzTbnkStmlDsl}.
As in Neumann \cite{bookvonneumann} and in Example 2.1 right after (2.3) in \cite{czgskublics}, the components of the (\emph{canonical $($extended normalized von Neumann})) \emph{ $d$-frame} 
\begin{equation}
\vec f=(\vec a,\vec c)=\bigl((a_1,\dots,a_d),\,(c_{i,j}:  i,j\in[d],\,i\neq j) 
\label{eq:sLskPnspTfXskQv}
\end{equation}
are the  $1$-dimensional subspaces $a_i=Fv_i\in V\in \Subp F V$ for $i\in [d]$ and $c_{i,j}=F(v_i-v_j)$  for $i\neq j\in[d]$ in $\Subp F V$. Thus, by \eqref{eq:mtzTbnkStmlDsl}, the components of $\vec f$  are the following points
\begin{gather}
a_i=[0,\dots,0,1,0\dots,0] \text{  for }i\in [d-1],\ a_d=[0,\dots,0,-1],
\label{eq:tdrmBchSnnRmNa}\\
\text{and } c_{i,j}=[0,\dots,0,1,0\dots,0,-1,0,\dots,0]    \text{ for }i\neq j\in[d],
\label{eq:tdrmBchSnnRmNb} 
\end{gather}
where the unit 1 is at the $i$-th position in both cases and the $-1$ is at the $j$-th position, in $\Sub\psd$. Note that $c_{i,j}=c_{j,i}$ for $i,j\in[d]$ distinct  but, according to \eqref{eq:tdrmBchSnnRmNb}, their \emph{canonical forms} are different\footnote{When we consider $c_{i,j}$ an element of $\coring j i$, to be defined soon, then we  use the canonical form given in  \eqref{eq:tdrmBchSnnRmNb}.}.

For $i,j,k\in[d]$ pairwise distinct,
repeating what von Neumann and his followers did but using the notation of  \cite[(2.5)]{czgskublics}, the $(i,j)$-th \emph{coordinate ring} of $L
$ with respect to  $\vec f$ is \begin{align}
\coring i j =\coring {a_i}{a_j}:=\set{x\in L: x\vee a_j=a_i\vee a_j,\,\, x\wedge a_j=0}.
\label{eq:coring}
\end{align}
To define the ring operations, we need the following \emph{projectivities} from Neumann \cite{bookvonneumann}; we use the visual notation from Cz\'edli and Skublics \cite{czgskublics}.
So for pairwise distinct parameters $p,q,r\in [d]$, let
\begin{align}
& \bpf p q r q\colon [0,a_p\vee a_q]\to[0,a_r\vee a_q],\quad
x\mapsto (x\vee c_{p,r})\wedge (a_r\vee a_q),
\label{eq:bprkTsd}\\
&\bpf p q p r\colon [0,a_p\vee a_q]\to[0,a_p\vee a_r],\quad
x\mapsto (x\vee c_{q,r})\wedge (a_p\vee a_r).
\label{eq:bpHrd}
\end{align}
For $i,j,k\in[d]$ pairwise distinct
and $x,y\in\coring i j$,  we let
\begin{align}
x\bplus i j k y&:=(a_i\vee a_j)\wedge \Bigl(\bigl( (x\vee a_k)\wedge (c_{i,k}\vee a_j)\bigr) \vee \bpf i j k j(y) \Bigr),
\label{eq:hGkfmVtLvKtlc}\\
x\bszor i j k y&:=(a_i\vee a_j)\wedge \Bigl(\bpf i j i k(x) \vee \bpf i j k j(y) \Bigr),\text{ and}
\label{eq:hGkfmVtLvKtld}\\
x\bminus i j k y&:=(a_i\vee a_j)\wedge 
\Bigl(  a_k\vee \bigl( (c_{j,k}\vee x)\wedge(a_j\vee \bpf i j i k(y))\bigr) \Bigr);
\label{eq:hsTrnbNsRlZlpN}
\end{align}
they are in $\coring i j$ and do not depend on $k$.
Except that the lattice polynomials defined in \eqref{eq:hGkfmVtLvKtlc}, \eqref{eq:hGkfmVtLvKtld}, and \eqref{eq:hsTrnbNsRlZlpN} as well as the projections defined in \eqref{eq:bprkTsd} and \eqref{eq:bpHrd} are 
\begin{equation}
\text{built from }\vee,\ \wedge,\text{ and the components of }\vec f,
\label{eq:lNgtZmkrkm}
\end{equation}
their details are not relevant here, and there are other ways to define appropriate $\oplus$, $\otimes$, and $\ominus$. In fact,
as Herrmann \cite[2 lines after Theorem 2.2]{herrmann84} notes, Neumann used the opposite of $\bszor i j k$. Fortunately, what we need from von Neumann's voluminous \cite{bookvonneumann}, has already been summarized in Herrmann \cite[Theorem 2.2]{herrmann84}, in Section 2 of Cz\'edli and Skublics, and (partially) in Freese \cite[Page 284]{freese}. Furthermore, the isomorphism given in \eqref{eq:mtzTbnkStmlDsl} allows us to pass from $\Sub{\vecspace F d}$ to $\Sub\psd$.
So,  based on \eqref{eq:coring}--\eqref{eq:hsTrnbNsRlZlpN}, we can recall the following theorem.

\begin{theorem}[von Neumann \cite{bookvonneumann} for $3\leq d\in\Nplu$ and Day and Pickering \cite{daypick} for $d=3$]\label{thm:kgSzsRtlBkkmHdkg} For $i,j\in[d]$ distinct, the operations defined in \eqref{eq:hGkfmVtLvKtlc}, \eqref{eq:hGkfmVtLvKtld}, and \eqref{eq:hsTrnbNsRlZlpN} in $L=\Sub\psd$ do not depend on $k\in[d]\setminus \set{i,j}$, and 
\[R(i,j)=\bigl(R(i,j);\bplus i j k, \bminus i j k, \bszor i j k\bigr)
\]
is a ring, called the $(i,j)$-th \emph{coordinate ring}, for each $k\in[d]\setminus \set{i,j}$. 
The  map $\delta_{d,1}\colon F\to\coring d 1$ defined by $\delta_{d,1}(r):=[r,0,\dots,0,-1]$  is a ring isomorphism (and so it is a field isomorphism). 
So is the map $\delta_{i,j}\colon F\to \coring i j$ defined by 
\begin{equation}
\delta_{i,j}(r):=[0,\dots,0,r,0,\dots,0,-1,0,\dots,0]\in \Sub\psd
\label{eq:tmBtmBHltBhBSkSwRs}
\end{equation}
with $r$ at the $j$-th position and $-1$ at the $i$-th position. Thus, the coordinate rings $\coring i j$,  $i\neq j\in[d]$, are all isomorphic to the field $F$. The elements $a_i$ and $c_{j,i}$ are the zero and the unit of $\coring i j$.
The ring isomorphisms given in \eqref{eq:tmBtmBHltBhBSkSwRs} commute\footnote{We compose maps from right to left; e.g., $(\alpha\beta)(x)=\alpha(\beta(x))$.} with the projectivities defined in \eqref{eq:bprkTsd} and \eqref{eq:bpHrd}, respectively.
 That is, 
for any $p,q,r\in[d]$ such that $|\{p,q,r\}|=3$, 
\begin{equation}
\bpf p q r q \circ \delta_{p,q} =\delta_{r,q}\ \text{ and }\ 
\bpf p q p r \circ \delta_{p,q} =\delta_{p,r}.
\label{eq:wtlVdRcpBrCrM}
\end{equation}
Furthermore,  using the superscript ${}^{\text{rest}}$ to denote the restrictions of the projectivities occurring in \eqref{eq:wtlVdRcpBrCrM} to $\coring p q$, 
\begin{gather}
\bpf p q r q^{\text{rest}}\colon\coring p q\to \coring r q \text{ is a ring isomorphism,}
\label{eq:sjNKbrhrNaMgLjr}\\
\text{so is } \bpf p q p r^{\text{rest}}\colon\coring p q\to \coring p r,
\label{eq:sjNKbrhbtlMgLjr}
\end{gather}
and \eqref{eq:wtlVdRcpBrCrM} remains true if we change the projections in it to their restrictions to $\coring p q$. 
\end{theorem}

\section{Proving Theorem \ref{thmnok}}\label{sect:aproof}

A \emph{generating vector} of a lattice $L$ is a vector $\vec b=(b_1,\dots,b_s)$ of not necessarily distinct elements of $L$ such that $\set{b_1,\dots,b_s}$  generates  $L$.

\begin{proof}[Proof of Observation \ref{obs:hbNvlnlgnD}]
We argue by way of contradiction. Suppose that $k$ is large enough in the given sense but $L^k$ has an $n$-dimensional generating vector 
$(\uid b 1$, \dots, $\uid b n)$. For $i\in [k]$, let $\pi_i\colon L^k\to L$ denote the $i$-th projection defined by $\vec x\mapsto x_i$. Let $\uid g i:=(\pi_i(\uid b 1),\dots,\pi_i(\uid b n)\in L^n$. 
As $k$ is large, there are $i,j\in [k]$ such that $i\neq j$ and $\uid g i \leq \uid g j$, understood componentwise. Then for any $n$-ary lattice term $f$, we have that 
\begin{equation}
\begin{gathered}
\pi_i\bigl(f(\uid b 1,\dots, \uid b n)\bigr)  =  
f\bigl(\pi_i(\uid b 1),\dots, \pi_i(\uid b n)\bigr) =   
f(\uid g i)\cr
\leq f(\uid g j) =
f\bigl(\pi_j(\uid b 1),\dots, \pi_j(\uid b n)\bigr) =\pi_j\bigl(f(\uid b 1,\dots, \uid b n)\bigr).
\end{gathered}
\label{eq:mMrdJklBtnZnm}
\end{equation}
As $(\uid b 1$, \dots, $\uid b n)$ is a generating vector, \eqref{eq:mMrdJklBtnZnm} implies that $\pi_i(\vec x)\leq \pi_j(\vec x)$ for every $\vec x\in L^k$, which is a contradiction completing the proof.
\end{proof}

\begin{proof}[Proof of Lemma \ref{lemma:bgzsLm}]
Since $V'\subseteq V$, \eqref{eq:mMtlnZknZltcwrTb} makes sense.
For a subset $H\subseteq  V'$, 
since the operation of spanning is order-preserving and idempotent,
\begin{equation*}
\fvgener F H\subseteq \fvgener F {\fvgener P H} \subseteq \fvgener F {\fvgener F H}=\fvgener F H
\end{equation*}
and $\fvgener F {\fvgener P H}=\phi(\fvgener P H)$ imply the last sentence of the lemma. 

Let $X$ be a subspace of $V'$, denote its dimension by $t$, and take a maximal subset $U:=\set{\uid a 1,\dots,\uid a t}$ of linearly independent vectors in $X$. Then, for $i\in[t]$, $\uid a i$ is of the form $\uid a i=(u_{i,1},\dots, u_{i,d})$ with entries from $P$, and the rank of the matrix $A:=(u_{i,j})_{t\times d}$ is $t$. As $U$ generates (in other words, linearly spans) $X$ in $V'$, the last sentence of the lemma gives that $Y:=\fvgener F U$ equals $\phi(X)$. The rank $t$ of $A$ is captured by determinants,  so it remains $t$ when we pass from $P$ to $F$. Hence,  $\phi(X)=Y$ is also of dimension $t$. 
Since both $V'$ and $V$ are of the same finite dimension $d$, it follows that $\phi$ is cover-preserving, $\phi(0)=0$, and $\phi(1)=1$. Denote the join in $\Subp P {V'}$ and that in  $\Subp F V$ by $\vee'$ and $\vee$, respectively. For $X,Y\in V'$, the last sentence of the lemma allows us to compute as follows:
\begin{align*}
\phi(X\vee' Y)&=\phi(\fvgener P{X\cup Y})= \fvgener F{X\cup Y}\cr
&=\fvgener F{\fvgener F X\cup \fvgener F Y} \cr
&=\fvgener F{\phi(X)\cup \phi(Y)}=\phi(X)\vee\phi(Y).
\end{align*} 
Thus, $\phi$ is a join-homomorphism.
We claim that if $X,Y\in \Subp P {V'}$ such that $\phi(X)\leq \phi(Y)$, then $X\leq Y$. Suppose the contrary, that is, $\phi(X)\leq \phi(Y)$ but $X\nleq Y$. Then $Y<X\vee'Y$ and  
$\phi(Y)=\phi(X)\vee \phi(Y)=\phi(X\vee' Y)$ together contradict  the fact that $\phi$ is dimension-preserving. Therefore, $X\leq Y\iff \phi(X)\leq Y$, that is, $\phi$ is an order-embedding. We know from Lemma 1 of Wild \cite{wild} that every cover-preserving order embedding between two lower semimodular lattices is a meet-embedding. Therefore, since subspace lattices are lower semimodular (in fact, they are even modular), we obtain that $\phi$ preserves the  meets. Thus, $\phi$ is a lattice embedding, completing the proof of Lemma \ref{lemma:bgzsLm}
\end{proof}

The following observation is trivial by definitions.

\begin{observation}\label{obs:nnRdRsfcmJntH} 
Let $F$ be a field, $3\leq d\in\Nplu$, and let
$\iud 1 u=[\nuid u 1_1,$ \dots, $\nuid u 1_d]$, \dots, $\iud k u=[\nuid u k_1$, \dots, $\nuid u k_d]$ be points in $\psd(F)$; according to our convention, we assume that $\{\nuid u 1_d,\dots,\nuid u k_d\}\subseteq\set{0, -1}$. 
Then a point $\vec v=[v_1,\dots, v_d]\in\psd(F)$, with $v_d\in\{0,1\}$ again, belongs to the subspace generated (in other words, spanned) by $\{\iud 1 u,\dots, \iud k u\}$ if and only if there exist $\lambda_1,\dots, \lambda_k\in F$ such that 
\begin{equation}
v_i=\sum_{j\in [k]} \lambda_j\nuid u j_i\ \text{ for }i\in[d].
\label{eq:nkRkmznMdtstlDbnGFNr}
\end{equation}
If $\vec v$ is a finite point, that is, if $v_d=-1$, then \eqref{eq:nkRkmznMdtstlDbnGFNr} implies that 
  $\Theta:=\{i:\nuid u i_d=-1\}\neq\emptyset$ and  $\sum_{i\in \Theta}\lambda_i=1$. 
If $\vec v$ and all the $\uid u {i}$, $i\in [k]$, are finite points, then \eqref{eq:nkRkmznMdtstlDbnGFNr} means that  $\vec v$ 
is a so-called \emph{affine combinations} of $\iud 1 u$, \dots, $\iud k u$, that is, $\sum_{i\in [k]}\lambda_i=1$.
\end{observation}

As $\coring d 1\cong F$ is a field, it is natural that we need the (partial) unary operation of forming reciprocals. By passing from   Huhn diamonds, see Huhn \cite{huhn}, to our setting based on (von Neumann) frames, such a unary operation could be derived from any of the two division operations given at the bottom of Page 510 in Day and Pickering \cite{daypick}.
However, while \cite{daypick} deals with a more general class of modular lattices, we need this unary operation only in the simple situation where our lattice is of the form $\Sub\psd$ and  $\coring d 1$ is determined by the canonical frame. Hence, and also because some details will be useful later, we define such a unary operation directly.  Namely, for $i,j,k\in[d]$ pairwise distinct and $x\in \Sub\psd$, we define
\begin{align}
\recip i j k(x):=  \Bigl(\Bigl(\Bigl(\bigl((x\vee c_{k,i})\wedge(a_j\vee a_k)\bigr)\vee c_{j,i}\Bigr)\,\Bigr)\cr
\wedge (a_k\vee a_i)\vee c_{k,j}\Bigr) 
\wedge(a_i\vee a_j)\in  \coring i j.
\label{eq:vzvmgBrnplF}
\end{align}

\begin{figure}[ht] 
\centerline{ \includegraphics[scale=1.0]{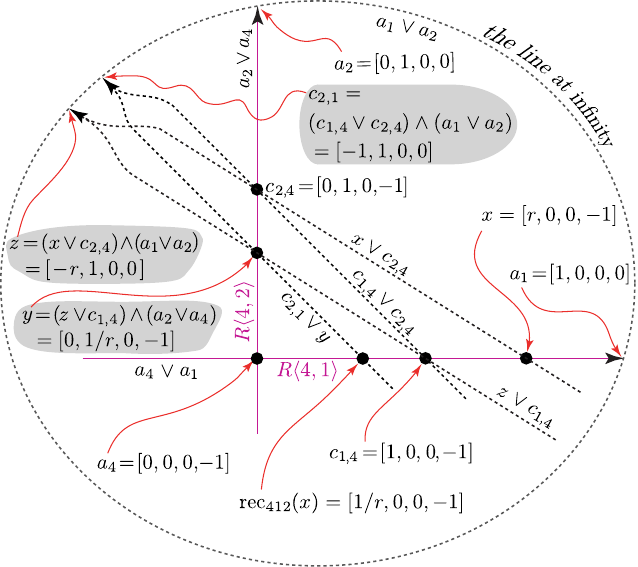}} \caption{Computing reciprocals}\label{figinv}
\end{figure}

\begin{lemma}\label{lemma:RcPrkWvk}
If $F$ is a field, $3\leq d\in\Nplu$, and $x\in \coring i j\subseteq \Sub\psd$ such that $x\neq a_i$, then $\recip i j k(x)$ is the reciprocal of $x$ in $\coring i j$, that is, $x\bszor i j k \recip i j k(x)=c_{j,i}$. Furthermore, $\recip i j k(a_i)=a_j$ and \eqref{eq:lNgtZmkrkm} is valid for \eqref{eq:vzvmgBrnplF}, too. (Note that $a_j\notin\coring i j$ and, 
by Theorem \ref{thm:kgSzsRtlBkkmHdkg}, $a_i$ and $c_{j,i}$ are the zero $0_{\coring i j}$ and the unit $1_{\coring i j}$ in $\coring i j$, respectively.)
\end{lemma}

\begin{proof} We deal only with  $(d,i,j,k)=(4,4,1,2)$, which reflects the general case.  The proof is given by Figure \ref{figinv}.  To exemplify how this figure determines an easy formal argument in a straightforward way, we present only the following details; similar details from other proofs will be omitted.
By Theorem \ref{thm:kgSzsRtlBkkmHdkg}, $x=\delta_{i,j}(r)=\delta_{4,1}(r)=[r,0,0,-1]$ for some $r\in F\setminus \set 0$, and  
it suffices to show that $\recip 4 1 2 (x)=\delta_{4,1}(1/r)$, that is,  $\recip 4 1 2 (x)= [1/r,0,0,-1]$.
With 
$z:=(x\vee c_{2,4})\wedge(a_1\vee a_2)$ and 
$y:=(z\vee c_{1,4})\wedge (a_2\vee a_4)$, we have that
$\recip 4 1 2(x)=(y\vee c_{2,1})\wedge (a_4\vee a_1)$.
Assuming that $z=[-r,1,0,0]$ is already known, we proceed to the next computation step. Namely,  we verify that $y$ is correctly given in the figure. 
Using $c_{1,4}=[1,0,0,-1]$, $c_{2,4}=[0,1,0,-1]$, $a_1=[1,0,0,0]$, $a_2=[0,1,0,0]$, and $a_4=[0,0,0,-1]$ from \eqref{eq:tdrmBchSnnRmNa}--\eqref{eq:tdrmBchSnnRmNb},
Observation \ref{obs:nnRdRsfcmJntH} implies 
that a point $P$ is  in $z\vee c_{1,4}$ if and only if it is of the form $[-\beta_1 r+\beta_2,\beta_1,0,-\beta_2]$ 
such that $\beta_1\in F$,  $\beta_2\in\set{0,1}$, and $(\beta_1,\beta_2)\neq(0,0)$. Similarly, $P$ is in $a_2\vee a_4$ if and only if it is of the form $[0,\lambda_1,0,-\lambda_2]$ such that  $\lambda_1 \in F$, $\lambda_2\in\set{0,1}$, and $(\lambda_1,\lambda_2)\neq(0,0)$. Comparing the two forms, we have that $\beta_1=\lambda_1$, $\beta_2=\lambda_2$, and $-\beta_1 r+\beta_2=0$. By the last equality and $r\neq 0$, we  have that $\beta_1\neq 0 \iff \beta_2\neq 0$. 
So $(\beta_1,\beta_2\neq (0,0)$ and $\beta_2\in\set{0,1}$ give that $\beta_2=1$. Hence,  $-\beta_1 r+\beta_2=0$ implies that $\beta_1=1/r$, and so $P=[0,1/r,0,-1]$.  This computation verifies the equality $y=[0,1/r,0,-1]$, confirming the figure. 
\end{proof}

\begin{proof}[Proof of Theorem \ref{thmnok}] In virtue of the isomorphism given in \eqref{eq:mtzTbnkStmlDsl}, we can assume that $L=\Sub{\psd(F)}=\Sub\psd$. Denoting the prime field of $F$ by $P$, let $L'=\Sub{\psd(P)}$. Let $\pvec f$ and $\vec f$ be the canonical frames in $L'$ and $L$ according to \eqref{eq:tdrmBchSnnRmNa}--\eqref{eq:tdrmBchSnnRmNb}, respectively. The isomorphism given in \eqref{eq:mtzTbnkStmlDsl} depends on the underlying field, this is why the next sentence indicates the corresponding fields in the subscripts. 
It follows from Lemma \ref{lemma:bgzsLm} and \eqref{eq:mtzTbnkStmlDsl} that for the composite map $\phi':=  \eta_F\circ\phi\circ\eta_P^{-1}$, we have that 
\begin{gather}
\phi'\colon L'\to L\text{ is a 0-, 1-, and cover-preserving  lattice embedding}
\label{eq:fktPtpTkfKtkVKpna}
\\
\text{and }\phi'(\pvec f)=\vec f,\text{ understood componentwise.}
\label{eq:fktPtpTkfKtkVKpnb}
\end{gather}
First, we deal with  the second inequality in \eqref{eq:nKmZfRtdrCh}. 
As $P$ is a prime field, we know from Gelfand and Ponomarev's result (see also lines 2--3 of page 494 in Z\'adori \cite{zadori2} or Section \ref{sect:append-ZLprfn} here)  that $L'$ is 4-generated. Pick a 4-dimensional generating vector $\vec g\ppr =(g'_1,g'_2,g'_3,g'_4)$ of $L'$ and, with $\phi'$ from \eqref{eq:fktPtpTkfKtkVKpna}, let
\begin{equation}
g_i:=\phi'(g'_i)\text{ for }i\in [4]; \text{ so }\phi'(\vec g\ppr)=(g_1,\dots,g_4).
\label{eq:ktrMfjJlLllnskmTtnJv}
\end{equation} 
Denote by $M$ and $m$ the denominator and the second summand occurring in \eqref{eq:nKmZfRtdrCh}, respectively. 
So $M=\lint{d^2/4}$ and $m=\uint{t/M}$.
Since $mM\geq t=\mng F$, there exist not necessarily distinct elements $r_{i,j}\in F\setminus\set 0$,  $i\in[m]$ and $j\in[M]$, such that $\{r_{i,j}: i\in[m]$ and $j\in [M]\}$ generates $F$ as a field.
Consider the following $\lint{d/2}$-by-$d$  ``pattern matrix''
\begin{equation}
A:=\begin{pmatrix}
\forall&0&0&\dots&0& 0& \forall& \dots &\forall& -1\cr
0&\forall&0&\dots&0& 0& \forall& \dots &\forall& -1\cr
0&0&\forall&\dots&0& 0& \forall& \dots &\forall& -1\cr
\vdots&\vdots&\vdots&\ddots&\vdots& \vdots& \vdots& \dots &\vdots& -1\cr
0&0&0&\dots&\forall& 0& \forall& \dots &\forall& -1\cr
0&0&0&\dots&0& \forall& \forall& \dots &\forall& -1\cr
\end{pmatrix}.
\label{eq:krBthmFpBltMJrnkTlZg}
\end{equation}
Using \eqref{eq:nNzImdrHlTrmLj}, we obtain that  $A$ contains exactly $M$  universal quantifiers.
For $i\in [m]$, we obtain a ``real matrix'' $A(i)$ from  $A$ by changing the universals quantifiers to $r_{i,1}$, \dots, $r_{i,M}$.  So each of the  $r_{i,1}$, \dots, $r_{i,M}$ occurs in $A(i)$ exactly once and it occurs at a place where $A$ contains a universal quantifier.
Each row of $A(i)$ consists of the coordinates of a finite point of $\psd=\psd(F)$; let $\uid u {i,1}$, \dots, $\uid u {i,\lint{d/2}}$ be the finite points corresponding to the rows of $A(i)$ in this way. 
For example, $r_{i,1}$, \dots, $r_{\uint{d/2}}$
are substituted into the first row of the pattern matrix to obtain the first row of $A(i)$ and so 
\begin{equation}
\uid u {i,1}=[r_{i,1},0,0,\dots,0,0,r_{i,2},\dots,r_{\uint{d/2}},-1].
\label{eq:hmVdJzBrHLsGJl}
\end{equation}
We often refer to the rows of $A(i)$ as points of $\psd$. For $i\in[m]$, let 
\begin{equation}
g_{4+i}\text{ be the subspace of }\psd\text{ spanned by }
\set{\uid u {i,1},\dots, \uid u {i,\lint{d/2}}}.
\label{eq:rstStmjSNmpnThLl}
\end{equation}
In other words, $g_{4+i}$ is the subspace of $\psd$ spanned by the rows of $A(i)$. 
For a subset $X$ of $L$, let $\latgen X$ denote the sublattice of $L$ that $X$ generates; we shorten $\latgen{\set{x_1,\dots,x_n}}$ to  $\latgen{x_1,\dots,x_n}$. Having \eqref{eq:ktrMfjJlLllnskmTtnJv} and \eqref{eq:rstStmjSNmpnThLl},
we claim that  $\vec g:=(g_1, g_2,\dots,g_{4+m})$
 is a generating vector of $L$,  that is, 
\begin{equation}
\text{letting }S_0:=\latgen{g_1, g_2, \dots,g_{4+m}},\text{ we claim that }S_0=L.
\label{eq:mRdszmrztmkslBbnt}
\end{equation}
Since  $\set{g_1,\dots, g_4}$ generates $\phi'(L')$, we have that $\phi'(L')\subseteq S_0$. Thus, with reference to \eqref{eq:sLskPnspTfXskQv}, \eqref{eq:tdrmBchSnnRmNa},  \eqref{eq:tdrmBchSnnRmNb}, and \eqref{eq:fktPtpTkfKtkVKpnb}, we have that
\begin{equation}
\text{the components of }\vec f\text{ are in }\latgen{g_1,\dots,g_4}\subseteq S_0. 
\label{eq:fjBlRkhrgBtgFRc}
\end{equation}
Let\footnote{For \emph{this} proof, working with $S_0$ would be sufficient. We introduce $S_1$ and later $S$, because $S$ will be referenced in Section \ref{sect:bproof}.} 
\begin{equation}
S_1:=\latgen{\set{g_{5},\dots,g_{4+m}}\cup\set{\text{the components of }\vec f\,}}.
\label{eq:zbzGfSblZNtjs}
\end{equation}
As it is clear from \eqref{eq:fjBlRkhrgBtgFRc}, 
to prove \eqref{eq:mRdszmrztmkslBbnt}, it suffices to show that $S_1$ equals $L$. As a first but a long step, we show that 
\begin{equation}
\coring d 1\subseteq S_1. 
\label{eq:fcNgrmgdRkksrmmsT}
\end{equation}
First we show that for all $i\in [m]$,   
\begin{equation}
\text{every row of }A(i) \text{, as a point of }\psd\text{ and an atom of  } L \text{, is in }S_1.
\label{eq:mnTljtntKzRPsRp}
\end{equation} 
By symmetry, it suffices to show that $\uid u{i,1}$ from \eqref{eq:hmVdJzBrHLsGJl} is in $S_1$. By \eqref{eq:zbzGfSblZNtjs},
\begin{equation}
a_1\vee a_{\lint{d/2}+1}\vee a_{\lint{d/2}+2}\vee \dots \vee \vee a_{d-1}\vee a_d\in S_1.
\label{eq:flKszKnPtcPrnpT}
\end{equation}
Observation \ref{obs:nnRdRsfcmJntH}, \eqref{eq:tdrmBchSnnRmNa}, and \eqref{eq:tdrmBchSnnRmNb} imply that the subspace in \eqref{eq:flKszKnPtcPrnpT} consists of the points of the form
$[x_1,0\dots,0,x_{\lint{d/2}+1},\dots ,x_d]$ where the components are in $F$, not all of them is 0,  and $x_d\in\set{0,-1}$. So when we form the meet of $g_{4+i}$ and the subspace \eqref{eq:flKszKnPtcPrnpT}, then 
the fact  that none of the  $r_{i,j}$'s in (the ``diagonal part'' of) $A(i)$ is $0$  and  Observation \ref{obs:nnRdRsfcmJntH} imply that this meet is  $\uid u {i,1}$. So $\uid u {i,1}\in S_1$, proving \eqref{eq:mnTljtntKzRPsRp}.

Next, we show that for all $(i,j)\in[m]\times [M]$, 
\begin{equation}
\delta_{d,1}(r_{i,j})=[r_{i,j},0,\dots,0,-1]\in S_1,
\label{eq:klnmzZfnBNmTgKHvml}
\end{equation}
where $\delta_{d,1}$ is taken from Theorem \ref{thm:kgSzsRtlBkkmHdkg}. To ease the notation, we show this only for $r_{i,2}$; we can obtain the set membership $\delta(r_{i,j})\in S_1$ for all $j\in[M]$ analogously or we can conclude it by symmetry. Letting  $\iota:=1+\lint{d/2}$, we know from \eqref{eq:hmVdJzBrHLsGJl} that $r_{i,2}$ is the $\iota$-th coordinate of $\uid u {i,1}$. So it follows from Observation \ref{obs:nnRdRsfcmJntH}, \eqref{eq:tmBtmBHltBhBSkSwRs},  and \eqref{eq:tdrmBchSnnRmNa}--\eqref{eq:tdrmBchSnnRmNb} that
\begin{align}
\delta_{d,\iota}(r_{i,2})=[&0,\dots,0,r_{i,2},0,\dots,0,-1]\cr
=(&a_\iota\vee a_d)\wedge\Bigl( \uid u i \vee\bigvee_{j\in[d-1]\setminus\set\iota} a_j\Bigr); 
\label{eq:mCsnMkRtn}
\end{align}
the validity of \eqref{eq:mCsnMkRtn} is also explained by Figure \ref{figpsd}. Indeed, the figure shows how to extract the homogeneous coordinate $u_\iota$ of a finite point $\vec u$ in the particular case when $d=4$ and $\iota=3$; this technique is applicable to $\vec u:=\uid u {i,1}$, too. The first meetand in  \eqref{eq:mCsnMkRtn} is the vertical magenta coordinate axis $a_3\vee a_4$ in the figure. The second meetand in  \eqref{eq:mCsnMkRtn} is the  the horizontal magenta hyperplane $\vvec u\vee a_1\vee a_2$ through $\vec u$. The meet of these two meetands is $\uid v 3=\delta_{4,3}(u_3)$, a copy of $u_\iota$ in the coordinate ring $\coring d\iota$.
Since $\uid u {i,1}$ is in $
S_1$ by \eqref{eq:mnTljtntKzRPsRp} and so are the 
$a_j$'s occurring in  \eqref{eq:mCsnMkRtn} by \eqref{eq:zbzGfSblZNtjs}, we obtain that  $\delta_{d,\iota}(r_{i,2})\in S_1$. By \eqref{eq:tmBtmBHltBhBSkSwRs},  $\delta_{d,\iota}(r_{i,2})\in\coring d\iota$.  As \eqref{eq:lNgtZmkrkm} mentions, the ring isomorphisms given in \eqref{eq:sjNKbrhrNaMgLjr} and \eqref{eq:sjNKbrhbtlMgLjr} are  composed from lattice operations and constants that are components of the frame $\vec f$ and so they are in $S_1$ by \eqref{eq:zbzGfSblZNtjs}.  Hence, $S_1$ is closed with respect to these isomorphisms, and we obtain the set membership part ``$\in$'' of 
\begin{equation}
\delta_{d,1}(r_{i,2}) = \bpf d \iota d 1(\delta_{d,\iota}(r_{i,2}))\in S_1.
\label{eq:swCnHksZnGphDsMT} 
\end{equation}
As the equality part follows from   \eqref{eq:wtlVdRcpBrCrM}, so  \eqref{eq:swCnHksZnGphDsMT} holds. Clearly, the argument above is applicable for any $j\in[M]$,  not just for $j=2$, since we can replace $\uid u{i,1}$ with the row of $A(i)$ that contains $r_{i,j}$. (Note that for $j=1$ we have that $\iota=1$ and  so \eqref{eq:wtlVdRcpBrCrM} is not needed.)
Therefore,  \eqref{eq:swCnHksZnGphDsMT} holds for any $j\in[M]$,  not only for $j=2$. That is, we have proved  \eqref{eq:klnmzZfnBNmTgKHvml}. Applying  \eqref{eq:lNgtZmkrkm}
to the field operations  \eqref{eq:hGkfmVtLvKtlc}, \eqref{eq:hGkfmVtLvKtld}, \eqref{eq:hsTrnbNsRlZlpN}, and \eqref{eq:vzvmgBrnplF}, we obtain that  $S_1$ is closed with respect to the field operations of $\coring d 1$. As the field isomorphism $\delta_{d,1}$ sends generating sets to generating sets,  \eqref{eq:klnmzZfnBNmTgKHvml} yields that $S_1$ contains a generating set of the field $\coring d 1$. The two just-mentioned facts imply \eqref{eq:fcNgrmgdRkksrmmsT}.

Next, with reference to let   \eqref{eq:sLskPnspTfXskQv}--\eqref{eq:tdrmBchSnnRmNb}, let
\begin{equation}
S:=\latgen{\set{\text{the components of the canonical frame}}\cup \coring d 1}.
\label{eq:mDgTpcwKllRNsTSW}
\end{equation}
We obtain from \eqref{eq:zbzGfSblZNtjs} and  \eqref{eq:fcNgrmgdRkksrmmsT} that $S\subseteq S_1$. 
Therefore, to prove that $S_1=L$ and so \eqref{eq:mRdszmrztmkslBbnt} holds, it suffices to show that 
 \begin{equation}
S\text{, defined in \eqref{eq:mDgTpcwKllRNsTSW}, equals }L=\Sub\psd.
\label{eq:lHntlkdnGnfhBln}
\end{equation}
For later reference, we note that our argument 
\begin{equation}
\text{proving
\eqref{eq:lHntlkdnGnfhBln} does not use Gelfand and Ponomarev's theorem,}
\label{eq:znTmvzhDzBtT}
\end{equation}
which has already been mentioned; see also Theorem \ref{thm:gelfand} in Section \ref{sect:append-ZLprfn}. 

Next, we aim to prove  \eqref{eq:lHntlkdnGnfhBln}. From  \eqref{eq:lNgtZmkrkm}, we know that  $S$ is closed with respect to the ring isomorphisms  in \eqref{eq:sjNKbrhrNaMgLjr} and \eqref{eq:sjNKbrhbtlMgLjr}. Thus, for any $i,j\in[d]$ such that $i\neq j$ and for any $r\in F$,  
\begin{equation}
\coring i j\subseteq S\text{ and so }
[0,\dots,0,1,0,\dots,0,-1,0,\dots,0]\in S,
\label{eq:jbKlCmnKtbchDKj}
\end{equation}
where $r$ and $-1$ are at the $j$-th position and the $i$-th positions, respectively. 
Since each element of $L$ is the join of finitely many atoms, it suffices to show that any projective point $\vec u=[u_1,\dots,u_d]$ belongs to $S$. Since  at least one of the homogeneous coordinates $u_1$, \dots, $u_d$ is nonzero, symmetry allows us to assume that $u_d\neq 0$. 
That is, by homogeneity, we assume that $u_d=-1$. 
Letting   $\uid v i=[0,\dots,0,u_i,0,\dots,0,-1]$  (where $u_i$ is  sitting in the $i$-th component) for $i\in[d-1]$, we have that $\uid v i\in \coring d i\subseteq S$ by \eqref{eq:jbKlCmnKtbchDKj}. Figure \ref{figpsd} visualizes the situation for $d=4$. In the figure, the black-filled elements are in $S$ by \eqref{eq:mDgTpcwKllRNsTSW} and \eqref{eq:jbKlCmnKtbchDKj}, and therefore so are the three depicted hyperplanes  containing the empty-filled $\vec u$. 
Among these three hyperplanes, one is adorned in green, another in magenta, and the third is filled with a floral pattern. (When translated to grayscale, the green plane appears lighter than its magenta counterpart.)
As $\vec u$ is the meet of the three hyperplanes, $\vec u\in S$ is clear when $d=4$. The same idea works for any $3\leq d\in \Nplu$; indeed, 
\begin{equation*}
\vec u:=\bigwedge_{i=1}^{d-1}  \Bigl(\uid v i\vee \bigvee_{j\in[d-1]\setminus \set i} a_j \Bigr) \in S
\end{equation*}
follows in a straightforward way by using Observation \ref{obs:nnRdRsfcmJntH}. Thus,  \eqref{eq:lHntlkdnGnfhBln} holds, implying \eqref{eq:mRdszmrztmkslBbnt} and the second inequality in \eqref{eq:nKmZfRtdrCh}.

Our argument to show the first inequality in \eqref{eq:nKmZfRtdrCh} is practically the same as that of Strietz \cite{strietz} for partition lattices\footnote{As partition lattices with more than five elements are not modular, we note that Lemma \ref{lemma:Dtwo}, quoted from  Wille \cite{willefm}, is valid even without assuming modularity. Lemma 4.1 from Cz\'edli \cite{czgQuogen}, a variant of the $D_2$-Lemma, would also suffice here.}. The key is Wille's $D_2$ Lemma:

\begin{lemma}[$D_2$-Lemma in Wille \cite{willefm}]\label{lemma:Dtwo}
If a subdirectly irreducible modular lattice with more than two elements is generated by $e_0,e_1,\dots,e_t$, then $e_0\vee \dots \vee e_{i-1}\geq e_i\wedge \dots \wedge e_t$ for every $i\in [t]$.
\end{lemma}
 
By \eqref{eq:mtzTbnkStmlDsl}, $L\cong \Subp F V$, where $V=\vecspace F d$. 
We know from the folklore that $\Subp F V$  is subdirectly irreducible. Having no reference to this fact at hand, we present an easy in-line proof here; some details of this proof will also be used later.  Let $a$ and $b$ be distinct atoms of $\Subp F V$, then $a=F \uid w 1$ and $b= F \uid w 2$ with the uniquely determined and linearly independent vectors $\uid w 1=(\nuid w 1_1,\dots,\nuid w 1_d)$ and $\uid w 2=(\nuid w 2_1,\dots,\nuid w 2_d)$ in $V$ such that $\nuid w 1_1+\dots+\nuid w 1_d=1$ and  $\nuid w 2_1+\dots+\nuid w 2_d=1$. Letting $c:=F(\uid w 1+\uid w 2)$, a trivial computation shows that $\set{0=a\wedge b, a,b,c,a\vee b}$ is a sublattice isomorphic to $M_3$, the 5-element modular lattice of length 2. Therefore, the (clearly) atomistic and modular lattice $\Subp F V$  is subdirectly irreducible by lines 4--5 in page 349 of Gr\"atzer \cite{ggfound}. For later reference, let us summarize what we have also obtained:

\begin{observation}\label{obs:HrmDnkrFnhmNtm}
For any two distinct atoms $a$ and $b$ of $\Subp F V$, $c$ defined above by $c:=F(\uid w 1+ \uid w 2)$ is a third atom,   $\set{0,a,b,c,a\vee b}$ is a sublattice of $\Subp F V$, and this sublattice is isomorphic to $M_3$.
\end{observation}

Returning to the proof of Theorem \ref{thmnok}, let us assume, to reach a contradiction, that $L=\Subp F V$ is generated by a subset $\set{e_0,e_1,e_2}$. Applying Lemma \ref{lemma:Dtwo},
we have that $e_0\geq e_1\wedge e_2$ and $e_0\vee e_1\geq e_2$. These two inequalities and those that we obtain from them  by permuting the generators imply that $\set{e_0,e_1,e_2}$ generates an $M_3$ sublattice, which is a contradiction showing that $\mng{L}\geq 4$.
We have verified both inequalities in \eqref{eq:nKmZfRtdrCh}, and the proof of Theorem \ref{thmnok} is complete.
\end{proof}

\section{Proving Theorem \ref{thmk}}\label{sect:bproof}

As a preparation for the proof of the second theorem, we prove the following easy lemma.

\begin{lemma}\label{lemma:fGtln}
Assume that $L_1,\dots,L_k$ are finitely generated lattices, $L=L_1\times\dots\times L_k$ is their direct product, and $\uid b 1=(\nuid b 1_1,\dots, \nuid b 1_k)$, \dots, $\uid b t =(\nuid b t_1,\dots, \nuid b t_k)$ are elements of $L$. Then $\set{\uid b 1,\dots,\uid b t}$ generates $L$ if and only if
\begin{enumerate}
  \item \label{lemma:afGtln} For each $i\in[k]$, $\set{\nuid b 1_i,\dots,\nuid b t_i}$ generates $L_i$, and
  \item \label{lemma:bfGtln} For each $i\in[k]$, there is a $t$-ary lattice term $f_i$ such that $f_i(\nuid b 1_i$, $\dots$, $\nuid b t_i)$ equals $1_i$, the top element of $L_i$, but for every $j\in[k] \setminus \set i$, 
$f_i(\nuid b 1_j$, \dots, $\nuid b t_j)$ equals $0_j$, the bottom element of $L_j$. 
\end{enumerate}
\end{lemma}

Visually, we can form a $k$-by-$t$ matrix with the $\uid b i$'s being the columns and we apply the terms $f_i$ to the rows of this matrix.

\begin{proof}
First of all, note that $1_i$ and $0_j$ in the lemma exist since $L_i$ and $L_j$ are finitely generated.
To prove the ``only if'' part, assume that $\set{\uid b 1,\dots,\uid b t}$ generates $L$. Since the $i$-th projection $L\to L_i$ defined by $(x_1,\dots,x_k)\mapsto x_i$ sends generating sets to generating sets, \eqref{lemma:afGtln} holds. So does \eqref{lemma:bfGtln}  since there is a lattice term $f_i$ such that $(0,\dots,0,1,0,\dots,0)\in L$ (with $1$ sitting at the $i$-th place) equals $f_i(\uid b 1,\dots,\uid b t)$.

To prove the ``if'' part, assume that \eqref{lemma:afGtln} and \eqref{lemma:bfGtln} hold, and let $\vec w=(w_1,\dots,w_k)\in L$. 
For each $i\in[k]$, \eqref{lemma:afGtln} allows us to pick a $t$-ary lattice term $g_i$ such that $g_i(\nuid b 1_i,\dots,\nuid b t_i)=w_i$ in $L_i$. Furthermore,  \eqref{lemma:bfGtln} yields
a $t$-ary lattice term $f_i$ such that $f_i(\nuid b 1_i,\dots,\nuid b t_i)=1_i$ but $f_i(\nuid b 1_j,\dots,\nuid b t_j)=0_j$ for all $j\in[k]\setminus\set i$. Then 
\[\vec w=\bigvee_{i\in [k]}\bigl(g_i(\uid b 1,\dots,\uid b t) \wedge f_i(\uid b 1,\dots,\uid b t)\bigr)\in\latgen{\uid b 1,\dots,\uid b t}
\]
completes the proof of Lemma \ref{lemma:fGtln}.
\end{proof}

The following well-known fact follows from, say, Vanstone and Oorschot \cite[Theorem 3.3]{vanstone}.

\begin{fact}\label{fact:slfDlt} For $d\in\Nplu$ and a field $F$, 
$\Sub{ \vecspace F d}$ is a selfdual lattice.
\end{fact}

\begin{proof}[Proof of Theorem \ref{thmk}] 
To ease the notation, let $h:=\lint{d/2}$ (``$h$'' comes from \ul half) and $r:=4+\uint{t/\lint{d^2/4}}$.  We know from \eqref{eq:nKmZfRtdrCh} in Theorem \ref{thmnok} that $L=\Subp F V$ has an $r$-dimensional generating vector.

First, we show \eqref{eq:fDnKnCspJhb}. By Remark \ref{rem:sbrStvgnBx}, it suffices to show that $L^\mu$ is $(1+r)$-generated. 
For $i\in \set{1,h}$, let $A_i$ be the set of $i$-dimensional subspaces of $V$, that is, $A_i$ is the set of elements of height $i$ in $L$. In particular, $A_1$ is the set of atoms of $L$ and $|A_h|=\mu$; see \eqref{eq:ckHlrgjCnpzgN}. Define a binary operation ``product''  on $A_1$ as follows:
For $a,b\in A_1$, let 
\begin{equation}
ab:=\begin{cases}
c\text{ defined in Observation \ref{obs:HrmDnkrFnhmNtm} }&\text{if }a\neq b\text{ and}\cr
0=0_L&\text{ if }a=b.
\end{cases}
\label{eq:lgndMdkMrGjrZrttT}
\end{equation}
This operation, denoted by concatenation, has precedence over the lattice operations.
Clearly, Observation \ref{obs:HrmDnkrFnhmNtm} implies the following.

\begin{fact}\label{fact:btnNmnjtkBfRcSl} For any $b,e\in A_1$, 
either $b\neq e$ and $\set{0,b,be,e,be\vee e}$ is a sublattice isomorphic to $M_3$, or $b=e$  and $be=0$; in both cases, $b\leq be\vee e$.
\end{fact}

Let $\vec g=(g_1,\dots,g_r)$ be a generating vector of $L$.  Let $u_1,\dots, u_\mu$ be a repetition-free enumeration of the elements of (the $\mu$-element) $A_h$. For $j\in[r]$, we 
define $\iud j b\in L^\mu$ as the constant vector $(g_j,g_j,\dots,g_j)$. We define a further vector, $\iud 0 b:=(u_1,u_2,\dots,u_\mu)\in L^\mu$. We claim that 
\begin{equation}
\Psi:=\set {\iud 0 b,\iud 1 b,\dots,\iud r b} \text{ generates }L^\mu.
\label{eq:krCznhnstfdTVns}
\end{equation}
Since $\latgen{u_i,g_1,\dots,g_r}=L$ for all $i\in[\mu]$,
$\Psi$ (apart from self-explanatory notational differences) satisfies \eqref{lemma:afGtln} of Lemma \ref{lemma:fGtln}.

Showing that $\Psi$ satisfies  \eqref{lemma:bfGtln} of Lemma \ref{lemma:fGtln}, too, needs more work.
For each $i\in[\mu]$, fix an $h$-element subset $S_i$ of $A_1$ such that $u_i=\bigvee\set {e: e\in S_i}$. Let $\vxi=(\xi_1$, \dots, $\xi_r)$ be a vector of variables, and let $\pxi$ stand for $(\xi_0$, $\xi_1$, \dots, $\xi_r)$. For each element $w$ of $L$, let us fix an $r$-ary lattice term 
$\ter w(\vvxi $ such that $\ter w(\vvec g)=w$. 
If $w=ab$, see \eqref{eq:lgndMdkMrGjrZrttT}, then $\ter w(\vvxi $ is written as $\ter{(ab)}(\vvxi $. We can fix a $d$-element subset $B$ of $A_1$ such that $1=1_L$ equals $\bigvee B$. For each $i\in[\mu]$, we define the following lattice term:
\begin{equation}
f_i(\pxi):=\bigvee_{b\in B}\biggl(\ter b(\vvxi \wedge
\bigwedge_{e\in S_i}  \Bigl( 
\ter{(be)}(\vvxi \vee \bigl(\xi_0\wedge \ter e(\vvxi \bigr) 
\Bigr) \biggr). 
\label{eq:szhvmklDsrbnB}
\end{equation}
Let $(\pjg):=(u_j,g_1,\dots,g_r)$. We need to show that $f_i(\pjg)=0_L$ if $j\neq i$ and it is $1_L$ if $j=i$.
For the meetand $\beta_{b,e}(\pxi):=\ter{(be)}(\vvxi \vee \bigl(\xi_0\wedge \ter e(\vvxi \bigr)$ occurring in \eqref{eq:szhvmklDsrbnB}, 
\begin{equation}
\beta_{b,e}(\pjg)=
\ter{(be)}(\vvec g) \vee \bigl( u_j  \wedge \ter e(\vvec g)  \bigr)
= be \vee \bigl(u_j  \wedge  e \bigr).
\label{eq:fgtKmnmGlckmvn}
\end{equation}

There are two cases to consider. First, assume that $j=i$. Then, for every $e\in S_i=S_j$, $e\leq u_j$ yields that $\beta_{b,e}(\pjg)=be\vee e$,  
whereby Fact \ref{fact:btnNmnjtkBfRcSl} implies that $\ter b(\vvec g)=b\leq \beta_{b,e}(\pjg)$.  Thus, the meet $\bigwedge_{e\in S_i}$ as a meetand in \eqref{eq:szhvmklDsrbnB} makes no effect and we obtain that $f_i(\pjg)=\bigvee_{b\in B}\ter b(\vvec g)= \bigvee_{b\in B} b=1_L$ if $j=i$, as required.

Second, assume that  $j\neq i$. Since $u_i=\bigvee S_i$ and $u_j$, belonging to the antichain $A_h$,  are incomparable, there is an $e\in S_i$ such that $e\nleq u_j$. 
For this atom $e$,  $u_j\wedge e$ in \eqref{eq:fgtKmnmGlckmvn} is $0_L$, whence $\beta_{b,e}(\pjg)=be$. Thus, each of the joinands of $\bigvee_{b\in B}$ in \eqref{eq:szhvmklDsrbnB} is (at most) $\ter b(\vvec g)\wedge \beta_{b,e}(\pjg) =b\wedge be=0$, no matter whether $b=e$ or $b\neq e$. Therefore, $f_i(\pjg)=0$ if $j\neq i$, as required. 
Hence,  $\Psi$ 
satisfies \eqref{lemma:bfGtln} of Lemma \ref{lemma:fGtln}, whereby we conclude  \eqref{eq:krCznhnstfdTVns}. Thus,  $\mng{L^k}\leq \mng{L^\mu} \leq 1+r=5+\uint{t/\lint{d^2/4}}$, proving \eqref{eq:fDnKnCspJhb}.

Next, we deal with the first equality in \eqref{eq:fDnKnCspJha}.
Let $M:=\lint{d^2/4}$ and $m:=\mng L$. For the sake of contradiction, suppose that 
\begin{equation}
m=\mng L<\uint{t/M}\quad \text{ (indirect assumption).}
\label{eq:flJnhtLcBblkps}
\end{equation}
Let $\set{g_1,\dots,g_m}$ be a generating set of $L=\Sub V=\Sub{\vecspace F d}$. For each $i\in [m]$, let $n_i$ be the dimension of the subspace $g_i$. Let us pick an $n_i$-by-$d$ matrix $B(i)$ over $F$ such that the rows of $B(i)$ form a basis of $g_i$.  After performing the Gauss--Jordan elimination to the rows of $B(i)$, these rows still form a basis of $g_i$. Hence, we can assume that $B(i)$ is in reduced row echelon form and the number $n_i$ of its rows equals its rank.
So for $j,\iota\in [n_i]$, the $\iota$-th element in the $j$-th row of $B(i)$ is $\delta_{j,\iota}$ (Kronecker delta).
Note that  $B(i)$ has the same shape as $A(i)$ in 
\eqref{eq:krBthmFpBltMJrnkTlZg} would have if we changed the universal quantifiers in the main diagonal to units (that is, to 1's). Let $H(i)$ stand for the set of those entries of $B(i)$ that differ from $0$ and $1$. This entries are in the last $d-n_i$ columns, so $|H(i)|\leq n_i(d-n_i)$. Hence, using \eqref{eq:nNzImdrHlTrmLj} and 
that the quadratic function $x\mapsto x(d-x)$ takes its maximum at $d/2$, it follows that 
$|H_i|\leq n_i(d-n_i)\leq M$. Letting $H:=H(1)\cup\dots\cup H(m)$, we have that $|H|\leq mM$. Observe that no matter whether $t/M$ is an infinite cardinal, an integer number, or a non-integer number, the indirect assumption \eqref{eq:flJnhtLcBblkps} and $m\in\Nnul$ imply that 
\begin{equation}
m<t/M, \text{ whereby }|H|\leq mM < t=\mng F.
\label{eq:mnKlhfRkNTnsRkzgGkr}
\end{equation}
Let $P$ denote the subfield of $F$ generated by $H$. By \eqref{eq:mnKlhfRkNTnsRkzgGkr},  $P\subset F$ (proper subfield). With  $V':=\vecspace P d$ and $L':=\Subp P{V'}$, $\phi$ from  Lemma \ref{lemma:bgzsLm} is a lattice embedding  $L'\to L$. 
For $i\in[m]$, using that $B(i)$ is also a matrix over $P$, let $g_i'$ be the subspace of $V'$ spanned by the rows of $B(i)$. Denoting the set of rows of $B(i)$ by $\nuid X i$, the last sentence of Lemma \ref{lemma:bgzsLm} gives that
\begin{equation}
\phi(g'_i)=\phi(\fvgener P{\nuid X i})   = \fvgener F{\nuid X i}=g_i
\label{eq:tZsplmnkPKpCWmLWcps}
\end{equation}
for $i\in[m]$. Since $\phi$ is an embedding, $\phi(L')$ is a sublattice of $L$. This sublattice includes the generating set $\set{g_i:i\in [m]}$ by \eqref{eq:tZsplmnkPKpCWmLWcps}. Thus, $\phi(L')=L$, implying that $\phi$ is surjective. Pick an element $r\in F\setminus P$. By the surjectivity of $\phi$, the 1-dimensional subspace  $S:=\fvgener F{\set{(r,1,1,\dots,1)}}\in L$
has a $\phi$-preimage $S'\in L$. Since $\phi$ is length-preserving by Lemma \ref{lemma:bgzsLm}, $S'$ is also 1-dimensional. So $S':=\fvgener P{\set{(p,q_2,q_3,\dots,q_d)}}$ for some $p,q_2,\dots,q_d\in P$. The last sentence of Lemma \ref{lemma:bgzsLm} yields a $\lambda\in F$ such that $(r,1,1,\dots,1)=\lambda(p,q_2,q_3,\dots,q_d)$. Comparing the second components, $\lambda=q_2^{-1}\in P$. Thus, the equality of the first components yields that $r=p\lambda\in P$, contradicting the choice of $r$. Now that the indirect assumption \eqref{eq:flJnhtLcBblkps} has lead to a contradiction, we have shown the first inequality in \eqref{eq:fDnKnCspJha}.  Remark \ref{rem:sbrStvgnBx} gives the second inequality, so we have proved \eqref{eq:fDnKnCspJha}.

\begin{figure}[ht] 
\centerline{ \includegraphics[width=0.98\textwidth]{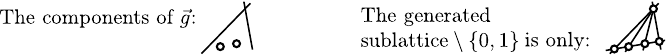}} \caption{For $\vec g=\uid g i$,  $\typ{\vvec g}$  cannot be $(2,2)$}\label{fig2}
\end{figure}

Next, we turn our attention to \eqref{eq:ktTlTkvrglLk}  and \eqref{eq:mtPhfGnprWmt}. So we assume that $F$ is a prime field and $d=2$. Furthermore, based on the isomorphism given in \eqref{eq:mtzTbnkStmlDsl}, let $\psd=\psp 2=\psp 2(F)$ be the projective plane over $F$ and, in the rest of the proof of Theorem \ref{thmk}, let  $L:=\Sub{\psp 2}$. Some geometric terms and methods in addition to the lattice theoretic ones  will frequently appear in our considerations. In particular, instead of drawing a usual Hasse diagram of $L=\Sub{\psp 2}$,  we visualize $L$ and its sublattices  by drawing the points and lines they  contain. 
Furthermore, we frequently use  the following definition (but only for projective \emph{planes}) without referencing it.

\begin{definition} For $L=\Sub{\psp 2}$ and a quadruple $\vec g=(g_1,\dots,g_4)\in L^4$, we say that $\vec g$ is in \emph{general position} if for any $\set{i,j,k}\subset[4]$ such that $|\set{ i,j,k }|=3$,
\begin{itemize}
  \item  $g_i\nleq g_j$, that is,  $\set{g_1,\dots,g_4}$ is an antichain;
  \item if $g_i$, $g_j$, and $g_k$ are points, then $g_i\nleq g_j\vee g_k$, that is, no three collinear points occur among the components of $\vec g$;  and 
  \item if $g_i$, $g_j$, and $g_k$ are lines, then $g_j\wedge g_j\nleq g_k$, that is, no three concurrent lines occur among the components of $\vec g$.
\end{itemize}
A \emph{complete quadrangle} is a quadruple $\vec g=(g_1,\dots,g_4)$ in general position such that $g_1$, \dots, $g_4$ are points.
\end{definition}

Analogously to an  earlier notation, $A_1$ is the set of points while $A_2$ is the set of lines. We show that 
\begin{equation}
\text{if }t=0\text{, }d=3\text{, and }\mng{L^k}=4\text{, then }k\leq 4.
\label{eq:lfGjbnfbhhM}
\end{equation}
So $F$ is a prime field now, and we can assume that $k$ is the largest positive integer such that $\mng{L^k}=4$. This makes sense since $k\geq 1$ by \eqref{eq:nKmZfRtdrCh} and the maximum exists by Observation \ref{obs:hbNvlnlgnD}. Choose a 4-dimensional generating vector $(\uid b 1,\dots, \uid b 4)$ of $L^k$. (Here the $\uid b i$, $i\in[4]$, are also vectors since they belong to $L^k$.)  Let 
\begin{gather}
\uid g i=(\nuid g i_1,\nuid g i_2,\nuid g i_3,\nuid g i_4):=  (\nuid b 1_i,\nuid b 2_i,\nuid b 3_i,\nuid b 4_i)\,\,\text{ for }i\in [k];\cr
\text{it is a generating vector of }L\text{ by Lemma \ref{lemma:fGtln}.}
\label{eq:BjtrBnkmstzStHnpr}
\end{gather}
Define the \emph{Kronecker delta} in a lattice $L$  by $\kdelta L i i:=1_L$ and, for $j\neq i$, $\kdelta L i j:=0_L$. 
Let $f_i$, $i\in [k]$, be the quaternary lattice terms provided by Lemma \ref{lemma:fGtln}; then 
\begin{equation}
f_i(\uid g j)= \kdelta L i j.
\label{eq:KrNckRdlt}
\end{equation}
As $\set{\nuid g i_1,\dots,\nuid g i_4}$
generates $L$, it is easy to see that for each $i\in [k]$ and $j\in[4]$, $\nuid g i_j$ is a point or a line. For later reference, we formulate this fact:
\begin{equation}
\nuid g i_j\notin\set{0, 1}\text{ and any 4-element generating set}\subseteq A_1\cup A_2.
\label{eq:mKzrmsMtpSt}
\end{equation}

\begin{figure}[ht] 
\centerline{ \includegraphics[width=0.98\textwidth]{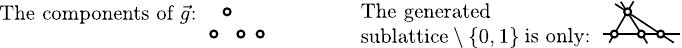}} \caption{A quadruple of points not in general position}\label{fig3}
\end{figure}

For $x\in L$, let $\hgh(x)$ denote the \emph{height} of $x$; it is the projective dimension plus 1. For example, for $x\in A_1$, $\hgh(x)=1$. 
For a generating vector $\vec g=(g_1,g_2,g_3,g_4)\in L^4$ of $L$, define the \emph{type} and the \emph{fine type} of $\vec g$ as 
\begin{gather*}
\typ{\vvec g}:=(|\set{i\in [4]: g_i\text{ is a point}}|,\,|\set{i\in [4]: g_i\text{ is a line}}|)\text{ and}\cr
\ftyp{\vvec g}:=(\hgh(g_1),\hgh(g_2),\hgh(g_3),\hgh(g_4)).
\end{gather*}
We know from \eqref{eq:mKzrmsMtpSt} that the sum of the  components of $\typ{\vvec g}$ and that of $\ftyp{\vvec g}$ are 4.  
It follows from \eqref{eq:BjtrBnkmstzStHnpr} and  \eqref{eq:mKzrmsMtpSt} that for every generating quadruple $\vec h$ and, in particular,  for every $i\in[k]$
\begin{equation}
\ftyp{\vec h}\in\set{1,2}^4 \text{ and  } \ftyp{\uid g i}\in\set{1,2}^4.
\label{eq:TrtnCcVjmpmVr}
\end{equation}
The \emph{type of a fine type}  $\vec \tau\in \set{1,2}^4$ is $\typ{\vec\tau}:=(|\{i\in[4]: \tau_i=1\}|$, $|\{i\in[4]: \tau_i=2\}|)$. Note the obvious rule:  $\typ{\ftyp{\uid g i}}=\typ{\uid g i}$ for every $i\in[k]$. 
Note also that our figures and arguments 
\begin{equation}
\text{will omit the most trivial cases like } \nuid g i_1=\nuid g i_2.
\label{eq:TrlPzGslRpbrkHsG}
\end{equation}
Using that every line contains at least three points,  Figure \ref{fig2}  shows that for any generating quadruple $\vec h$ and, in particular, for $i\in [k]$,
\begin{equation}
\text{neither }\typ{\vec h}\text{ nor }\typ{\uid g i} \text{ can be }(2,2).
\label{eq:prSkRdgKfStprTjn}
\end{equation}

In $\psp 2$,  any two distinct lines intersect in a point.  The following fact is also well known; see, for example, Veblen and Young \cite[page 93]{veblenyoung}. 
\begin{fact}\label{fact:fnlPvdlFhLWlfr}
If $\vec x=(x_1,\dots,x_4)$ and $\vec x'=(x'_1,\dots,x'_4)$ are complete quadrangles in $\psp 2$, then $\psp 2$  has an automorphism $\phi$ such that $\phi(x_i)=x'_i$ for $i\in [4]$. Consequently, $L$ also has such an automorphism. 
\end{fact}
Therefore, our figures are sufficiently general. We claim the following.

\begin{figure}[ht] 
\centerline{ \includegraphics[scale=1.0]{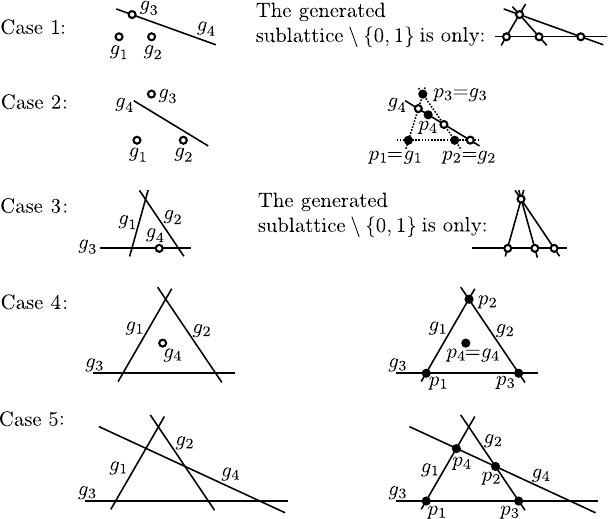}} \caption{
Proving Fact \ref{fact:tmgsfjtkgmnDhmkPsSr}}\label{fig4}
\end{figure} 

\begin{fact}\label{fact:nlLlslTshbr}
Every generating quadruple of $L$ is in general position.
\end{fact}

\begin{figure}[ht] 
\centerline{ \includegraphics[scale=1.0]{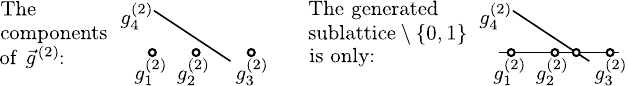}} \caption{Three collinear points and a line}\label{fig-noncoll}
\end{figure} 

To show this, assume that $\vec h$ is a generating quadruple. Since $\typ{\vec h}\neq(2,2)$ by \eqref{eq:prSkRdgKfStprTjn} and $L$ is selfdual, see Fact \ref{fact:slfDlt}, we can assume that $\typ{\vec h}\in\set{(4,0),(3,1)}$. If $\typ{\vec h}=(4,0)$, then $\vec h$ is in general position by  Figure \ref{fig3} and \eqref{eq:TrlPzGslRpbrkHsG}. For  $\typ{\vec h}=(3,1)$, we draw the same conclusion from Case 1 of Figure \ref{fig4}
and Figure \ref{fig-noncoll}. Thus, Fact \ref{fact:nlLlslTshbr} holds.

Our next step is to show the following fact.

\begin{fact}\label{fact:tmgsfjtkgmnDhmkPsSr} If $|F|\geq 3$, then  for  each generating vector $\vec g=(g_1,g_2,g_3,g_4)$ of $L$, there is a complete quadrangle $(p_1,\dots,p_4)$ of $L$ such that $p_i\leq g_i$ for $i\in[4]$.
\end{fact}

To show Fact \ref{fact:tmgsfjtkgmnDhmkPsSr}, observe that 
Facts  \ref{fact:fnlPvdlFhLWlfr} and \ref{fact:nlLlslTshbr} take care of the case $\typ{\pvec g }= (4,0)$.
Hence,  there are five cases to consider, see Figure \ref{fig4}, but each of them is obvious. We exclude Cases 1 and 3 since then $\set{g_1,\dots,g_4}$ does not generate $L$; indeed, the figure shows on the right what the generated sublattice is and this sublattice is clearly not the whole $L$ since every line of the projective plane has at least three\footnote{We now have at least four points since $|F|\geq 3$. However, we continue to use the term ``at least three points'' to make this argument applicable also when $|F|=2$.} points. In Cases 2, 4, and 5, the figure shows how to choose the $p_i$'s. Note for later reference that only Case 2 needs the assumption that $|F|\geq 3$, which makes it possible to pick a fourth point on the line $g_4$.
So, Figure \ref{fig4} has proved Fact \ref{fact:tmgsfjtkgmnDhmkPsSr}.

Now we can show that 
\begin{equation}
\text{if } \typ{\uid g i}\in\set{(4,0),\,(0,4)}\text{ for some }i\in[k]\text{, then } k=1.
\label{eq:mVnlTfGlnm}
\end{equation}
For the sake of contradiction, suppose that, say,  $\typ{\uid g 1}\in\set{(4,0),(0,4)}$ but $k>1$. By the selfduality of $L$, see  
Fact \ref{fact:slfDlt},  we 
can assume that $\typ{\uid g 1}=(4,0)$. 
First, we assume that $|F|\geq 3$. 
Fact \ref{fact:tmgsfjtkgmnDhmkPsSr} yields a
complete quadrangle $\vec p$ such that $p_i\leq \nuid g 2_i$ for $i\in[4]$. 
By Fact \ref{fact:nlLlslTshbr},  $\uid g 1$ is a complete quadrangle. Thus, 
by Fact \ref{fact:fnlPvdlFhLWlfr}, we can take an automorphism $\phi$ of $L$ such that $\phi(\uid g 1_i)=p_i\leq \nuid g 2_i$ for $i\in [4]$; we  write
$\phi(\uid g 1)\leq \uid g 2$ for short.
Using \eqref{eq:KrNckRdlt} and the fact that $f_1$ is order-preserving, we obtain that
\begin{equation}
1=\phi(\kdelta L 1 1)=\phi(f_1(\uid g 1))=f_1(\phi(\uid g 1))\leq f_1(\uid g 2)=\kdelta L 1 2=0,
\label{eq:fltskBrDskRcdlv}
\end{equation} 
which is a contradiction showing \eqref{eq:mVnlTfGlnm} for the case $|F|\geq 3$.

If $|F|=2$ and so the projective plane is the Fano plane, then the argument for \eqref{eq:mVnlTfGlnm} needs the following modifications. 
Even though Case 2 of Figure \ref{fig4} and Fact \ref{fact:tmgsfjtkgmnDhmkPsSr} fail for the Fano plane, Fact \ref{fact:tmgsfjtkgmnDhmkPsSr}  still holds for the 
particular case $\typ{\vvec g}\in\set{(1,3),\,(0,4)}$ since then the earlier argument relies only on Cases 3, 4, and 5 of Figure \ref{fig4}. 
Like we did right after \eqref{eq:mVnlTfGlnm}, we assume that \eqref{eq:mVnlTfGlnm} is false and its failure is witnessed by 
$\uid g 1$ of type $(4,0)$ and $\uid g 2$. If $\typ{\uid g 2}\in\set{(1,3),\,(0,4)}$, then the just-mentioned particular case of Fact  \ref{fact:tmgsfjtkgmnDhmkPsSr} leads to a  contradiction in the same way as before. We know from \eqref{eq:prSkRdgKfStprTjn} that $\typ{\uid g 2}\neq(2,2)$. 
If $\typ{\uid g 2}=(4,0)$, then Facts  \ref{fact:fnlPvdlFhLWlfr} and \ref{fact:nlLlslTshbr} give an automorphism $\phi\colon L\to L$ such that $\uid g 2=\phi(\uid g 1)$, whereby \eqref{eq:fltskBrDskRcdlv} (with equality in its middle rather than an inequality) leads to a contradiction.  Hence, based on \eqref{eq:mKzrmsMtpSt}, we can assume that $\typ{\uid g 2}=(3,1)$. 
Since, for any $i,j\in[k]$, $\kdelta L i j$ is a fixed point of every automorphism of $L$, it follows that for any
system $(f_i:i\in[k])$ of quaternary lattice terms and for any family $(\psi_{i,j}: i,j\in[k])$ of automorphisms of $L$,
\begin{equation}
\text{\eqref{eq:KrNckRdlt} holds if an only if }
f_i\bigl(\psi_{i,j}(\uid g j)\bigr)=\kdelta L i j\text{ for all }i,j\in[k].
\label{eq:hJcmnhhXtvrn}
\end{equation}

\begin{figure}[ht] 
\centerline{ \includegraphics[scale=1.0]{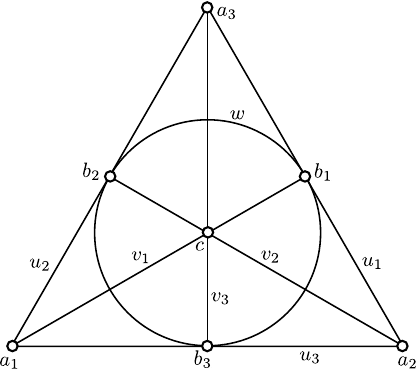}} \caption{Notations for the Fano plane}\label{fig6}
\end{figure}

Figure \ref{fig6} shows how we denote the points and the lines of the Fano plane; they belong to $L$ and $|L|=16$.
By Fact \ref{fact:nlLlslTshbr}, $\uid g 2$ is in general position. Thus, by symmetry and  \eqref{eq:hJcmnhhXtvrn}, we can assume that $\uid g 2=(a_1,a_2,a_3,w)$; see Figure \ref{fig6}.
By  Fact \ref{fact:fnlPvdlFhLWlfr} and  \eqref{eq:hJcmnhhXtvrn}, we can also assume that $\uid g 1=(a_1,a_2,a_3,c)$.
To define a subset $S$, let us agree that sets of the forms $\set{x_i:i\in[3]}$ and $\{x_{i,j}:i,j\in[3]$, $i\neq j\}$ will simply be denoted by $\set{x_i}$ and  $\set{x_{i,j}}$, respectively. These sets consist of three and six elements, respectively. With these temporary notations, we let
\begin{equation}
\begin{aligned}
S&:=\ul{\set{(a_i,a_i)}} \cup \set{(u_i,u_i)} \cup
 \set{(b_i,u_i)} \cup \set{(0,a_i)} \cr
&\cup\set{(0,b_i)} \cup
\set{(0,u_i)} \cup \set{(0,v_i)} \cup \set{(a_i,1)} \cup
\set{(b_i,1)} 
\cr
&\cup \set{(u_i,1)} \cup   \set{(v_i,1)} \cup \set{(a_i,v_i)} 
\cup  \set{(a_i,u_j)}\cr
& \cup \set{\ul{(c,w)}, (0,0),(1,1),(c,1),(0,w),(w,1),(0,1),(0,c)};
\end{aligned}
\label{eq:sBPzrvncVkwrv}
\end{equation}
the underlined terms of \eqref{eq:sBPzrvncVkwrv} will occur in \eqref{eq:kfRsZtlCGbDvgb}.
It is straightforward  to check\footnote{\label{foOt:mapLe}Alternatively, an appropriate  program in Maple V (version 5.9, 1997, Waterloo Maple Inc.) 
is presented in (the Appendix) Section \ref{sect:maple}; it is also
is available from the author's website  \href{http://tinyurl.com/g-czedli/}{http://tinyurl.com/g-czedli/} .} that $S$ is a sublattice of $L^2$. This fact and \eqref{eq:KrNckRdlt} imply that
\begin{align}
(1,0)&=\bigl(\kdelta L 1 1,\kdelta L 1 2 \bigr)=
\bigl(f_1(\uid g 1),f_1(\uid g 2)\bigr)\cr
&=
\bigl(f_1(a_1,a_2,a_3,c),\, f_1(a_1,a_2,a_3,w)\bigr)\cr
&= \bigl(f_1(a_1,a_1),f_1(a_2,a_2),f_1(a_3,a_3),f_1(c,w)\bigr)\in S,
\label{eq:kfRsZtlCGbDvgb}
\end{align}
which contradicts \eqref{eq:sBPzrvncVkwrv}. 
Hence, \eqref{eq:mVnlTfGlnm} holds even if $|F|=2$, that is, it holds for all prime fields.

\begin{figure}[ht] 
\centerline{ \includegraphics[width=0.98\textwidth]{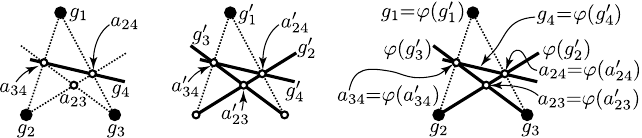}} \caption{Proving Fact \ref{fact:thNckbsRnmRsmgRmhzNf}}\label{fig5}
\end{figure}

Next, for fine types $(\xi_1,\xi_2,\xi_3,\xi_4)$ and $(\eta_1,\eta_2,\eta_3,\eta_4)$, let us say that they are \emph{complementary} if $\xi_i+\eta_i=3$ for all $i\in[4]$. \eqref{eq:TrtnCcVjmpmVr} sheds more light on this concept.

\begin{fact}\label{fact:thNckbsRnmRsmgRmhzNf} If there are $\vec g,\pvec g \in \set{\uid g i: i\in [k]}$ such that $\typ{\vvec g}=(3,1)$ and $\typ{\pvec g }=(1,3)$, then $k=2$ and, furthermore,  $\ftyp{\vvec g}$ and $\ftyp{\pvec g }$ are complementary.
\end{fact}

To show Fact \ref{fact:thNckbsRnmRsmgRmhzNf} by way of contradiction, assume that $\vec g,\pvec g \in \set{\uid g i: i\in [k]}=:\Gamma$ such that $\typ{\vvec g}=(3,1)$ and $\typ{\pvec g }=(1,3)$ but $\ftyp{\vvec g}$ and $\ftyp{\pvec g }$ are not complementary. We know from Fact \ref{fact:nlLlslTshbr} that $\vec g$ and $\pvec g$ are in general position.  Apart from permutations, $\ftyp{\vvec g}=(1,1,1,2)$ and $\ftyp{\pvec g }=(1,2,2,2)$; see Figure \ref{fig5}. 
The left of Figure \ref{fig5} shows how to define three auxiliary points; for example (in the language of $L$), $a_{24}:=(g_1\vee g_3)\wedge g_4$ and $a_{23}:=(a_{34}\vee g_3)\wedge (a_{24}\vee g_2)$; similarly for the middle of the figure. 
It is straightforward to see that if $(g_1,a_{23},a_{24},a_{34})$ 
was not in general position then neither $\vec g$ would be, and similarly for $(g'_1,a'_{23},a'_{24},a'_{34})$ in the middle of Figure \ref{fig5}. Hence, Fact \ref{fact:fnlPvdlFhLWlfr} yields an automorphism $\phi$ of $L$ such that  $\phi(g'_1)=g_1$, $\phi(a'_{23})=a_{23}$, $\phi(a'_{24})=a_{24}$, and $\phi(a'_{34})=a_{34}$;  
 see on the right of Figure \ref{fig5}.
As the figure shows, $\vec g\leq \phi(\pvec g )$, understood componentwise. In other words, $\phi^{-1}(\vvec g)\leq \pvec g$. As $\vec g$ and $\pvec g$ are in $\Gamma=\set{\uid g i: i\in[4]}$, we can assume that $\uid g 1=\vec g$ and $\uid g 2= \pvec g $. So  $\phi^{-1}(\uid g 1)\leq \uid g 2$. Hence \eqref{eq:fltskBrDskRcdlv}, with $\phi^{-1}$ instead of $\phi$, gives  contradiction.
This shows that 
\begin{equation}
\ftyp{\vvec g}\text{ and }\ftyp{\pvec g }\text{ are complementary, as required.}
\label{eq:zCnrwMvJnWBhPZ}
\end{equation}

Next, we show that for any fine type $\vec \tau$,
\begin{equation}
\text{there is at most one }\vec h\in\Gamma\text{ such that }\vec\tau=\ftyp{\vec h}.
\label{eq:fktbkPtkpFtkPg}
\end{equation}
To verify \eqref{eq:fktbkPtkpFtkPg}, we can assume that 
$\typ\tau\neq (2,2)$ since otherwise \eqref{eq:fktbkPtkpFtkPg} is clear by  \eqref{eq:prSkRdgKfStprTjn}.
So let 
$\vec h, \vec h'\in \Gamma$ such that  $\vec\tau=\ftyp{\vec h}=\ftyp{\vec h'}$; we need to show that $\vec h=\vec h'$. 
If $\vec\tau\in\set{(4,0),(0,4)}$, then $\vec h=\vec h'$ is clear by \eqref{eq:mVnlTfGlnm}. Out of the cases $\typ\tau=(3,1)$ and $\typ\tau=(1,3)$, it suffices to settle the first one since then the other follows by duality; see Fact \ref{fact:slfDlt}.  As the components of $\vec\tau$ share a symmetrical role, we can assume that $\vec\tau=\ftyp{\vec h}=(1,1,1,3)$; see Case 2 in Figure \ref{fig4} with $\vec g$ instead of $\vec h$. No problem if $|F|=2$, as $p_4$ (the fourth point on $g_4$) is not needed here.
On the right of Case 2 in the figure, the bottom left black-filled point, the bottom right black-filled point, the middle empty-filled point, and the top left empty-filled point, in this order, form a 
complete quadrangle $\vec z$. Indeed, if $\vec z$ was not in general position, then neither  $\vec h$  would be and so $\vec h$ would contradict Fact \ref{fact:nlLlslTshbr}. 
Observe that $\vec z$ determines $\vec h$. Hence, applying Fact \ref{fact:fnlPvdlFhLWlfr} to $\vec z$ and to the analogously defined quadruple determining $\vec h'$, Fact \ref{fact:fnlPvdlFhLWlfr} implies that $\vec h'=\phi(\vec h)$ for some automorphism $\phi$ of $L$. Hence, $\vec h'=\vec h$ in this case since otherwise \eqref{eq:fltskBrDskRcdlv} (with notational changes and equality instead of inequality in the middle)  would lead to a contradiction. We have shown \eqref{eq:fktbkPtkpFtkPg}. 

Next, continuing the argument for Fact \ref{fact:thNckbsRnmRsmgRmhzNf}, 
assume that $\vec h\in \Gamma$. By  \eqref{eq:prSkRdgKfStprTjn} and \eqref{eq:mVnlTfGlnm},  $\typ{\vec h}\notin\{(4,0)$, $(0,4)$, $(2,2)\}$. Hence, $\typ{\vec h}=(3,1)=\typ{\vvec g}$ or $\typ{\vec h}=(1,3)=\typ{\pvec g }$. Since $L$ is selfdual by Fact \ref{fact:slfDlt} (or since the second alternative needs almost the same treatment), we can assume that $\typ{\vec h}=(3,1)=\typ{\vvec g}$. Then $\vec h\in \Gamma$ and $\vec g\in \Gamma$ have the same role.  Hence \eqref{eq:zCnrwMvJnWBhPZ} applies to $\vec h$ and $\pvec g $, whence $\ftyp{\vec h}$ and $\ftyp{\pvec g }$ are complementary. As only one fine type is complementary to $\ftyp{\pvec g }$, we have that $\ftyp{\vec h}=\ftyp{\vvec g}$. Thus, \eqref{eq:fktbkPtkpFtkPg} yields that $\vec h=\vec g$. 
So $\vec h=\vec g\in \set{\vec g,\pvec g }$, implying that $k=2$ and completing the proof of Fact \ref{fact:thNckbsRnmRsmgRmhzNf}.

Next, assume that $k>2$. We know from \eqref{eq:prSkRdgKfStprTjn} and  \eqref{eq:mVnlTfGlnm} that, for all $i\in[k]$,  $\typ{\uid g k}\notin\set{(4,0),\,(2,2),  \,(0,4)}$. So $\typ{\uid g 1}\in\set{(3,1),\,(1,3}$. By duality, we can assume that $\typ{\uid g 1}=(3,1)$. 
As  Fact \ref{fact:thNckbsRnmRsmgRmhzNf} together with $k>2$ exclude that $\typ{\uid g i}=(1,3)$ for some $i\in[k]\setminus\set 1$, 
we have that $\typ{\uid g i}=(3,1)$ for all $i\in[k]$.
Hence, for every $i\in[k]$, $\ftyp{\uid g i}$ is one of the fine types $(1,1,1,2)$, $(1,1,2,1)$, $(1,2,1,1)$, and $(2,1,1,1)$. 
Since each of these four fine types occurs at most once by \eqref{eq:fktbkPtkpFtkPg}, it follows that $k\leq 4$, proving \eqref{eq:lfGjbnfbhhM}.

Clearly,  \eqref{eq:lfGjbnfbhhM}, the first inequality in \eqref{eq:nKmZfRtdrCh},  and  the particular $(t,d)=(0,3)$ case of (the already proven) \eqref{eq:fDnKnCspJhb} and  \eqref{eq:lfGjbnfbhhM} imply  \eqref{eq:mtPhfGnprWmt}.

Next, interrupting the proof of Theorem \ref{thmk}, we recall 
and, for the reader's convenience, prove the following lemma; its first part follows from known deep results.

\begin{lemma}[Day and Pickering \cite{daypick}, Herrmann \cite{herrmann84}, Hermmann and Huhn \cite{herrmannhuhn}]\label{lemma:mGszRzbTr}
Every complete quadrangle $\vec p=(p_1,p_2,p_3,p_4)$ in $\psp 2$ (the projective plane over the \emph{prime} field $F$) is a generating vector of $L=\Sub{\psp 2}$. So is 
every quadruple $\vec q$ in general position such that $\typ{\vvec q}\neq (2,2)$.
\end{lemma}

In the context of this paper, the proof of Lemma \ref{lemma:mGszRzbTr} is straightforward and, what is important in Section \ref{sect:append-ZLprfn}, it does not rely on  Gelfand and Ponomarev's result, which was mentioned after \eqref{eq:RsPlWrQn}. Here, we provide a concise demonstration.
\eqref{eq:TrtnCcVjmpmVr} shows that  the assumption $\typ{\vvec q}\neq (2,2)$ cannot be omitted from the lemma.

\begin{figure}[ht] 
\centerline{ \includegraphics[width=0.98\textwidth]{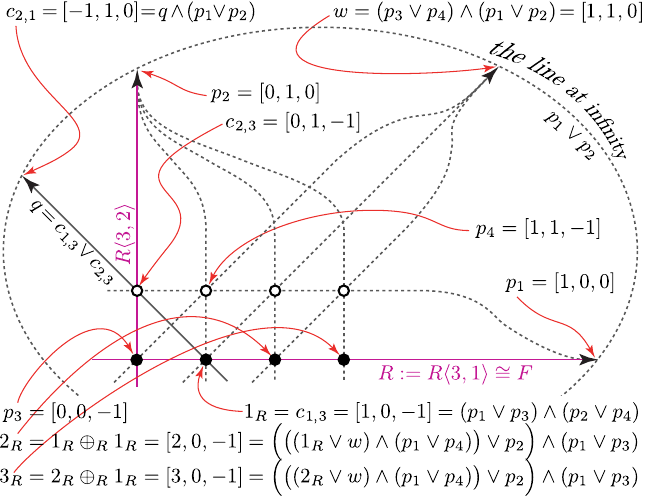}} \caption{Generating the (subspace lattice of the) projective plane}\label{figipplndmnd}
\end{figure}

\begin{proof}[Proof of Lemma \ref{lemma:mGszRzbTr}]
Let $\vec p$ be a complete quadrangle. By Fact \eqref{fact:fnlPvdlFhLWlfr}, we can assume that $\vec p$ is the canonical complete quadrangle; see Figure \ref{figipplndmnd}. Let $S:=\latgen{p_i:i\in[4]}$. The figure shows  
that the elements  of the canonical von Neumann $3$-frame, $a_i:=p_i$ for $i\in [3]$ and $c_{i,j}=c_{j,i}$ for $i\neq j\in[3]$, are in $S$. In particular, 
$1_{\coring 3 1}=c_{1,3}\in S$. As $\coring 3 1\cong F$ by Theorem \ref{thm:kgSzsRtlBkkmHdkg}, $\coring 3 1$ is a prime field and so it is generated by $1_{\coring 3 1}$. Therefore, since $S$ is closed with respect to the field operations by \eqref{eq:lNgtZmkrkm}, $\coring 3 1 \subseteq S$. In virtue of \eqref{eq:znTmvzhDzBtT}, we can apply \eqref{eq:lHntlkdnGnfhBln} to conclude that $S=L$, as required.
This proves the first half of Lemma \ref{lemma:mGszRzbTr}.

To show the second half,   \eqref{eq:TrtnCcVjmpmVr}, the first half of Lemma \ref{lemma:mGszRzbTr}, and duality allow us to assume that $\typ{\vvec q}=(3,1)$. We can assume that 
$q_1,q_2,q_3$ are points and $q_4$ is a line. 
Letting $\vec q$ play the role of $\vec g$ on the left of Figure \ref{fig5}, we obtain that $\set{a_{24},a_{34}}\subseteq \latgen{q_1,\dots,q_4}=:S$. So $S$ contains a complete quadrangle, $(q_2,q_3,a_{24},a_{34})$, whereby the first part of the lemma implies that $S=L$, as required. 
We have proved Lemma \ref{lemma:mGszRzbTr}. 
\end{proof}

To complete the proof of Theorem \ref{thmk},
we need to show \eqref{eq:ktTlTkvrglLk}. With its assumptions, if $\mng {L^k}\leq 3$, then Remark \ref{rem:sbrStvgnBx} would give that $\mng L\leq 3$, contradicting \eqref{eq:nKmZfRtdrCh}.
Hence, $\mng {L^k}\geq 4$. By Remark \ref{rem:sbrStvgnBx}, it suffices to prove that $L^4$ has a $4$-element generating set. Let $e$ be a line and 
$a,b,c$ be three non-collinear points of the projective plane 
such that none of these points lies on $e$. Then the quadruple $(e,a,b,c)$ is in general position;  think of the left of Figure \ref{fig5} and $(e,a,b,c,e):=(g_4,g_1,g_2,g_3)$.) 
Keeping the explanatory sentence right after Lemma \ref{lemma:fGtln} in mind, take the matrix
\begin{equation*} U= (u_{i,j})_{4\times 4}:=
\begin{pmatrix}
e & a & b & c \cr
a & e & b & c \cr
a & b & e & c \cr
a & b & c & e \cr
\end{pmatrix} ,
\end{equation*}
and let $\uid g i=(u_{i,1}, u_{i,2}, u_{i,3}, u_{i,4})$ be the $i$-th row of $U$ for $i\in[4]$. 
With $\vxi=(\xi_1$, $\xi_2$, $\xi_3$, $\xi_4)$ as a vector of variables, define the following quaternary lattice terms for $i,j\in [4]$, $i\neq j$:
\begin{align}
w_i(\vvxi &:=\bigwedge_{j\in [4]\setminus \set{i}} (\xi_i\vee \xi_j), \cr
\hij i j(\vvxi &:=\xi_j\wedge\bigwedge_{s\in[4]\setminus\set{i,j}}\bigl(w_i(\vvxi \vee \xi_s \bigr), \text{ and}\cr
\fer_i(\vvxi &:=\bigvee_{j\in [4]\setminus \set{i}}
\hij i j(\vvxi .
\label{eq:bwnRblmRDcGlh}
\end{align}
The superscript (e) of $f_i$ will be a useful reminder later. Some substitution values of these terms are given as follows:
\[
\begin{tabular}{l|r|r|r|r|r|r|r|r|r|r|r|r|r|r}
$\xi_1$ & $\xi_2$ & $\xi_3$ & $\xi_4$ & $w_1(\vvxi$ & $\hij 1 2(\vvxi$ & $\hij 1 3(\vvxi$ & $\hij 1 4(\vvxi$ & $\fer_1(\vvxi$   \cr
\hline\hline
$e$&$a$&$b$&$c$   &$1$ &$a$ &$b$ &$c$ & $1$\cr
\hline
$a$&$e$&$b$&$c$   &$a$ &$0$ &$0$ &$0$ & $0$\cr
\hline
$a$&$b$&$e$&$c$   &$a$ &$0$ &$0$ &$0$ & $0$\cr
\hline
$a$&$b$&$c$&$e$   &$a$ &$0$ &$0$ &$0$ & $0$\cr
\hline
\end{tabular}
\]
The last column above shows that $\fer_1(\uid g j)=\kdelta L 1 j$.
By symmetry or by three additional similar tables, 
\begin{equation}
\fer_i(\uid g j)=\kdelta L i j\text{ holds for all }i,j\in[4]. 
\label{eq:vtVnmNtfgTRnsjbh}
\end{equation}
Note for later reference that all we needed to prove \eqref{eq:vtVnmNtfgTRnsjbh} is only that 
\begin{equation}
\ftyp{a,b,c,e}=(1,1,1,2)\text{ and }(a,b,c,e)\text{ is in general position.}
\label{eq:GttKdKrcZRspRn}
\end{equation}
By  \eqref{eq:vtVnmNtfgTRnsjbh}, Condition \eqref{lemma:bfGtln} of Lemma \ref{lemma:fGtln} holds. So does  Condition \eqref{lemma:afGtln} of the same lemma by the second half of Lemma \ref{lemma:mGszRzbTr}.
Thus, the columns of $U$ form a $4$-element generating set of $L^4$ by Lemma \ref{lemma:fGtln},  completing the proof of \eqref{eq:ktTlTkvrglLk} and that of  Theorem \ref{thmk}.
\end{proof}

\section{Proving Theorem \ref{thm:prd} and Example \ref{expl:mZvsDnStSg}}\label{sect:cproof}

\begin{proof}[Proof of Theorem \ref{thm:prd}] 
If $\lambda$ is infinite, then $|L|=2^{\aleph_0}$ and so $L$ is not finitely generated. (In fact, it is not even $\aleph_0$-generated.) 
If a prime field $F$ occurred at least five times in the direct product \eqref{eq:tSvnRnkDglhlg} and $L$ was $4$-generated, then $\Sub{\vecspace F 3}^5$   would also be $4$-generated by Remark \ref{rem:sbrStvgnBx}, contradicting \eqref{eq:mtPhfGnprWmt}. Thus, the condition right after \eqref{eq:tSvnRnkDglhlg}  is necessary. 
The rest of the proof assumes this condition. We need to prove that $\mng L=4$. In fact, it suffices to find an at most 4-element generating set since the assumption $\lambda\neq 0$ together with \eqref{eq:nKmZfRtdrCh}  and Remark \ref{rem:sbrStvgnBx} imply that $\mng L\geq 4$. 
Furthermore, by Remark \ref{rem:sbrStvgnBx} again, we can assume that each prime field occurs exactly four times. 
So, taking \eqref{eq:mtzTbnkStmlDsl} also into account, we assume that 
\[
L=\prod_{i\in [k]}\prod_{\nu\in[4]}L_{i,\nu}
\text{, where }L_{i,\nu}=\Sub {\psp 2{(F_i)}},\ \ F_i\ncong F_j
\]
for $i\neq j$, and we  construct an (at most) 4-element generating set of $L$. 

For $i\in [k]$, let $\inud i p_1$,  $\inud i p_2$, $\inud i p_3$, and  $\inud i p_4$ be the points (and also the atoms in the corresponding subspace lattice)
$[1,0,0]$, $[0,1,0]$, $[0,0,-1]$, and  $[1,1,-1]$ in the projective plane $\ipsp i:=\psp2(F_i)$ over $F_i$, respectively; see Figure \ref{figipplndmnd} where the superscript $i$ is never indicated. 
Let $\inud i c_{2,3}:=(\inud i p_1\vee \inud i p_4)\wedge (\inud i p_2\vee \inud i p_3)$.   
Figure \ref{figipplndmnd} shows how we define $\inud i c_{1,3}$, $\inud i c_{2,1}$, and (for later use) $\inud i w$. We let 
\[
\inud i q:=\inud i c_{1,3}\vee \inud i c_{2,3} \  \text{ and }\    
\iud i r:=(\inud i p_1,\inud i p_2,\inud i p_3,\inud i q).
\]
Figure \ref{figipplndmnd} shows and it is easy to verify that 
\begin{equation}
\inud i p_4=\Bigl(\bigl((\inud i p_1\vee \inud i p_3)\wedge \inud i q\bigr)\vee \inud i p_2\Bigr) \wedge
\Bigl(\bigl((\inud i p_2\vee \inud i p_3)\wedge \inud i q\bigr)\vee \inud i p_1\Bigr). 
\label{eq:hvnBFjfktkpprhXbj}
\end{equation}
For $i\in[k]$, we define the following four quadruples:
\begin{align}
&\biud i 1 r:=(\inud i q,\inud i p_1,\inud i p_2,\inud i p_3),
&&\biud i 2 r:=(\inud i p_1,\inud i q,\inud i p_2,\inud i p_3)
\label{eq:armLpfRskjkZ}\\
&\biud i 3 r:=(\inud i p_1,\inud i p_2,\inud i q,\inud i p_3),
&&\biud i 4 r:=(\inud i p_1,\inud i p_2,\inud i p_3,\inud i q)=\iud i r.
\label{eq:brmLpfRskjkZ}
\end{align}
Form a $([k]\times[4])$-by-$4$ matrix from these vectors as row vectors. 
So the rows of this matrix are indexed by pairs taken from $[k]\times[4]$ and there are four columns. The $(i,\nu)$-th row of the matrix is $\biud i\nu r$.
We claim that the four columns of the matrix generate $L$. To prove this, we need to verify both conditions given in Lemma \ref{lemma:fGtln}. The satisfaction of Condition
\eqref{lemma:afGtln} of Lemma \ref{lemma:fGtln} follows from 
the second half of  Lemma \ref{lemma:mGszRzbTr}; it also follows from \eqref{eq:hvnBFjfktkpprhXbj} and the first half of  Lemma \ref{lemma:mGszRzbTr}.

Let $\vxi$ stand for the vector $(\xi_1,\xi_2,\xi_3,\xi_4)$ of variables.
To show that Condition  \eqref{lemma:bfGtln} of  Lemma \ref{lemma:mGszRzbTr} also holds and to complete the proof of the theorem, 
it suffices to define quaternary lattice terms $\finu i\nu=\finu i\nu(\vvxi$  for $(i,\nu)\in[k]\times[4]$ such that for any $(j,\kappa)\in[k]\times[4]$,
\begin{equation}
\finu i\nu(\biud j\kappa r)=
\begin{cases}
1_{L_j},\text{ if }(j,\kappa)=(i,\nu),\cr
0_{L_j},\text{ if }(j,\kappa)\neq(i,\nu).
\end{cases}
\label{eq:kzslgVczn}
\end{equation}
The term $\finu i\nu$ that we define is of the form
\begin{equation}
\finu i\nu(\vvxi := \ginu i\nu(\vvxi \wedge \fer_\nu(\vvxi\text{,  where }\fer_\nu \text{ is taken from }\eqref{eq:bwnRblmRDcGlh}.
\label{eq:nJfrlglSnX}
\end{equation}
(The superscript ``(e)" in \eqref{eq:nJfrlglSnX} comes from ``earlier''.)
Note that almost all of the terms we define in the rest of the proof are quaternary terms on $\vxi$ but $\vxi$ will often be dropped.
As the components in \eqref{eq:armLpfRskjkZ}--\eqref{eq:brmLpfRskjkZ} are permuted cyclically, we do the same with the variables of $\ginu i\nu$. 
So we define, in several steps, $\gif$;  then, in harmony with  \eqref{eq:armLpfRskjkZ}--\eqref{eq:brmLpfRskjkZ}, the rest of the terms $\ginu i\nu$ are given by the following rules:
\begin{align}
\ginu i 1(\vvxi &:=\gif(\xi_4,\xi_1,\xi_2,\xi_3),\label{eq:aprmtlGt}\\
\ginu i 2(\vvxi &:=\gif(\xi_1,\xi_4,\xi_2,\xi_3), \text{ and}\label{eq:bprmtlGt}\\
\ginu i 3(\vvxi &:=\gif(\xi_1,\xi_2,\xi_4,\xi_3),\label{eq:cprmtlGt}
\end{align}

Keeping an eye on Figure \ref{figipplndmnd}, $R=\inud iR=:\coring 3 1$ will also stand for $F_i$. In the figure, $\inud i 0_R:=0_{\inud i R}$, $\inud i 1_R=c_{1,3}$, $\inud i 2_R=[2,0,-1]$, and $\inud i 3_R=[3,0,-1]$ are already given. (As we have already mentioned, $i$ is not indicated in the figure.)
For all $s\in\Nplu$, we 
 defined $\inud i s_R\in L_i$ by induction as follows:
\begin{gather}
\inud i {(s+1)}_R := \inud i s_R\oplus_R \inud i 1_R 
=\Bigl(\bigl(
(\inud i s_R\vee \inud i w)\wedge (\inud i p_1\vee \inud i p_4)
\bigr)
\vee \inud i p_2\Bigr)\wedge (\inud i p_1\vee \inud i p_3).
\label{eq:fGlgmtSwHtlp}\\
\text{Clearly, for all }s\in\Nplu,\text{ we have that }
\inud i s_R=[s,0,-1]\in L_i;
\label{eq:sChphgLsrLpv}
\end{gather}
this follows also from Theorem \ref{thm:kgSzsRtlBkkmHdkg}.
When defining lattice terms for a given  $i\in [k]$, 
$\ater i c$ and $\bter i c$ denote terms closely related to a point $c\in\ipsp i$; we usually drop $i$ if such a term does not depend on it. 
First, to get rid of $\inud i p_4$ and bring $\inud i q$ in, we replace $\inud i p_4$ with the right-hand side of \eqref{eq:hvnBFjfktkpprhXbj} in every expression in Figure \ref{figipplndmnd}. In harmony with \eqref{eq:hvnBFjfktkpprhXbj},   \eqref{eq:fGlgmtSwHtlp}, and Figure \ref{figipplndmnd}, we let
\begin{align}
&\nater p_4=\nater p_4(\vvxi:=\Bigl(\bigl((\xi_1\vee \xi_3)\wedge \xi_4\bigr)\vee \xi_2\Bigr) \wedge
\Bigl(\bigl((\xi_2\vee \xi_3)\wedge \xi_4\bigr)\vee \xi_1\Bigr), 
\cr
&\nater w=\nater w(\vvxi:=(\xi_3\vee \nater p_4)\wedge (\xi_1\vee \xi_2),\,\,\  \nater p_\nu=\nater p_\nu(\vvxi:=\xi_\nu\text{ for }\nu\in[3],
\label{eq:bkspZdsDkKhRdtn}\\
&\nater 0=\nater 0(\vvxi:=\xi_3,\,\text{ and for }s\in \Nnul,
\label{eq:ckspZdsDkKhRdtn}
\\
&\nater {(s+1)}=\nater {(s+1)}(\vvxi\cr
&\phantom{\nater {(s+1)}} :=\Bigl(\bigl(
(\nater s\vee \nater w)\wedge (\xi_1\vee \nater p_4) \bigr)
\vee \xi_2\Bigr)\wedge (\xi_1\vee \xi_3).
\label{eq:dkspZdsDkKhRdtn}
\\
&\text{Let }\,
\nater c_{1,3}=\nater c_{1,3}(\vvxi:=\nater 1
\text{ and }\,
\nater c_{2,3}=\nater c_{2,3}(\vvxi:=\nater 1(\xi_2,\xi_1,\xi_3,\xi_4).\ 
\label{eq:ekspZsPrdsDkhRdtn}
\end{align}
So $\nater 0, \nater 1,\nater 2,\dots$ are lattice terms, not numbers. Comparing \eqref{eq:fGlgmtSwHtlp}, \eqref{eq:sChphgLsrLpv},  \eqref{eq:ckspZdsDkKhRdtn}, and \eqref{eq:dkspZdsDkKhRdtn}, we obtain that for all $j\in [k]$ and $s\in\Nnul$,  
\begin{equation}
\nater s(\iud j r)=[r, 0, -1]=:\inud j r_R \in L_j.
\label{eq:mNpvpGnRdbBn}
\end{equation}
By construction and since the subscripts 1 and 2 share a symmetrical role,  for any $j\in[k]$  and $\iota\in[4]$,
\begin{equation}
\nater p_\iota(\iud j r)=\inud j p_\iota,\ 
\nater w(\iud j r)=\inud j w,\ 
\nater c_{1,3}(\iud j r)= \inud j c_{1,3}
,\ 
\nater c_{2,3}(\iud j r)= \inud j c_{2,3}.
\label{eq:ksFhmHdfBtBrz}
\end{equation}
To define further terms, we need to distinguish between two cases.

First, assume that $t_i:=|F_i|$ is a prime number. We let 
\begin{align}
&\bter i p_3=\bter i p_3(\vvxi:=\nater p_3 \wedge \nater  {(t_i)},
\label{eq:h3wxWxTr}
\\
&\bter i p_1=\bter i p_1(\vvxi:=\ater i p_1\wedge(\bter i p_3 \vee \ater i p_{2} \vee \ater i p_{4}),
\cr
&\bter i p_2=\bter i p_2(\vvxi:=\ater i p_2\wedge(\bter i p_3 \vee \ater i p_{1} \vee \ater i p_{4}),\text{ and}
\cr
&\bter i p_4=\bter i p_4(\vvxi:=\ater i p_4\wedge(\bter i p_3 \vee \ater i p_{1} \vee \ater i p_{2}).
\notag 
\end{align}
We claim that for all  $\iota\in[4]$ and $j\in [k]$, 
\begin{equation}\text{in the lattice } L_j,\quad
\bter i p_\iota(\iud j r)=\begin{cases}
\inud j p_\iota,&\text{ if }j=i,\cr
0_{L_j}&\text{ if }j\neq i.
\end{cases}
\label{eq:gThkFtmgMtBlsK}
\end{equation}
To show this, observe that we know from \eqref{eq:mNpvpGnRdbBn} and \eqref{eq:ksFhmHdfBtBrz} that  both $\nater p_3 (\iud j r)=\inud j 0_R$ and $\nater  {(t_i)}(\iud j r)=\inud j {(t_i)}_R$ are points on the solid (magenta) horizontal line $\inud j p_3\vee \inud j p_1$ in Figure \ref{figipplndmnd}. If $j\neq i$, then $F_j\ncong F_i$,
$\inud j 0_R \neq \inud j {(t_i)}_R$,  and  the meet of these two distinct points is $\bter i p_3(\iud j r)=\emptyset=0_{L_j}$. If $j=i$, then 
$\inud j 0_R$ and $\inud j {(t_i)}_R$ are equal, whereby their meet is  $\bter i p_3(\iud j r)=\inud j 0_R =\inud j p_3$. This shows the validity of \eqref{eq:gThkFtmgMtBlsK} for $\iota=3$. 
Based on \eqref{eq:ksFhmHdfBtBrz} and Figure \ref{figipplndmnd}, we conclude  \eqref{eq:gThkFtmgMtBlsK}  from its particular case $\iota=3$.

Second, we assume that $F_i=\mathbb Q$, the field of rational numbers. Everything goes in the very same way as in the previous case when $F_i$ was finite except that \eqref{eq:h3wxWxTr} and the corresponding argument for the $\iota=3$ case of \eqref{eq:gThkFtmgMtBlsK} need some modifications. As a preparation to this task, with self-explanatory substitutions and using the terms \eqref{eq:bkspZdsDkKhRdtn}--\eqref{eq:ekspZsPrdsDkhRdtn}, we turn
 \eqref{eq:vzvmgBrnplF} with $(i,j,k)=(3,1,2)$ into the quinary lattice term
\begin{equation*}
\begin{aligned}
\nater{\recip 3 1 2}(x,\vvxi:= \Bigl(&\Bigl(\Bigl(\bigl((x\vee \nater c_{2,3})\wedge(\nater p_1\vee \nater p_2)\bigr)\vee \nater c_{1,3}\Bigr)\Bigr) 
\cr
&\wedge (\nater p_2\vee \nater p_3)\vee \nater c_{2,1} \Bigr) \wedge(\nater p_3\vee \nater p_1).
\end{aligned}
\end{equation*}
With  $T:=\{|F_j|: j\in[k]$ and $F_j$ is finite$\}$, let  
\begin{equation}
\bter i p_3=\bter i p_3(\vvxi:=\nater p_3\wedge\bigwedge_{t\in T}\Bigl(\nater p_1 \vee \nater{\recip 3 1 2}(\nater  t\bigl(\vvxi\bigr) \Bigr).
\label{eq:kTdnCTlnQGdnw}
\end{equation}
We claim that \eqref{eq:gThkFtmgMtBlsK} for $\iota=3$ still holds. 
If $i=j$, then  $F_j\cong \mathbb Q$ and for every $t\in T$, $\nater t(\iud j r)=\inud j t_R$ is not the zero element of $\inud j R\cong \mathbb Q$ by \eqref{eq:mNpvpGnRdbBn}. Hence, Lemma \ref{lemma:RcPrkWvk} implies that $\nater{\recip 3 1 2}(\nater  t(\iud j r))= \nater{\recip 3 1 2}(\inud j t_R)=
\inud j{(1/t)}_R$ belongs to $\inud j R$. In particular, $\inud j{(1/t)}_R$ is distinct from $\inud j p_1$, the infinite point of the (solid magenta) horizontal axis. 
This fact and the first equality in \eqref{eq:ksFhmHdfBtBrz} yield that the join in \eqref{eq:kTdnCTlnQGdnw} turns into $\inud j p_1\vee \inud j{(1/t)}_R$, which is  the (magenta) solid horizontal line in Figure \ref{figipplndmnd}. As this line contains $\nater p_3(\iud j r)=\inud j p_3$, we have that  $\bter i p_3(\iud j r)=\inud j p_3$ for $j=i$, as required.

Now let us examine what happens if $j\neq i$. 
Then the prime number $t:=|F_j|$ is in $T$ and the join 
$\nater p_1 \vee \nater{\recip 3 1 2}(\nater  t\bigl(\vvxi\bigr)$ is one of the meetands in \eqref{eq:kTdnCTlnQGdnw}. By  \eqref{eq:ksFhmHdfBtBrz},  $\nater  t(\iud j r)=\inud j t_R=\inud j 0_R=\inud j p_3$. We know from Lemma \ref{lemma:RcPrkWvk} that 
 $\recip  3 1 2(\inud j p_3)$ is $\inud j p_1$. Thus, using  \eqref{eq:ksFhmHdfBtBrz} again,  the meetand $\nater p_1 \vee \nater{\recip 3 1 2}(\nater  t\bigl(\vvxi\bigr)$ turns into $\inud j p_1\vee \inud j p_1=\inud j p_1$ when $\iud j r$ is substituted for $\vxi$. Since $\nater p_3$ turns into $\inud j p_3$ after the substitution and $\inud j p_3\wedge \inud j p_1=0_{L_j}$, we have that 
$\bter i p_3(\iud j r)=0_{L_j}$, as required.
We have shown that \eqref{eq:gThkFtmgMtBlsK} for $\iota=3$ still holds. Based on \eqref{eq:ksFhmHdfBtBrz}, we conclude  \eqref{eq:gThkFtmgMtBlsK}  from its particular case $\iota=3$. 

We have seen that  not matter if $F_i$ is finite or not,  \eqref{eq:gThkFtmgMtBlsK} holds for all $i,j\in[k]$ and $\iota\in[4]$.
This allows us to let
\begin{align}
&\gif(\vvxi:=\bigvee_{\iota\in[4]}\bter i p_\iota(\vvxi; \label{eq:bmTnkntJstRhmstrgtkr}\\
&\text{ then \eqref{eq:aprmtlGt}, \eqref{eq:bprmtlGt}, and \eqref{eq:cprmtlGt} define }\ginu i\nu(\vvxi\text{ for }\nu\in[3].
\notag
\end{align}
Since the ``rotational symmetry''  of \eqref{eq:armLpfRskjkZ}--\eqref{eq:brmLpfRskjkZ} and that of 
\eqref{eq:aprmtlGt}, \eqref{eq:bprmtlGt}, and \eqref{eq:cprmtlGt} 
correspond to each other, it suffices to verify \eqref{eq:kzslgVczn} only for $\nu=4$. So we are examining $\fif(\biud j\kappa r)=\gif(\biud j\kappa r)\wedge \fer_4(\biud j\kappa r)$; see \eqref{eq:nJfrlglSnX}.

First, assume that $(j,\kappa)=(i,4)$. 
Then the definition of $\fer_4$ in \eqref{eq:bwnRblmRDcGlh} does not depend on the underlying field and neither the argument showing \eqref{eq:vtVnmNtfgTRnsjbh} does, whence it follows from \eqref{eq:vtVnmNtfgTRnsjbh} and \eqref{eq:GttKdKrcZRspRn} that  $\fer_4(\biud j\kappa r)= \fer_4(\iud i  r) = 1_{L_i}=1_{L_j}$. 
All the joinands in \eqref{eq:bmTnkntJstRhmstrgtkr} are the respective points by  \eqref{eq:gThkFtmgMtBlsK}. As these points are in general position, we have that $\gif(\biud j\kappa r)=1_{L_j}$.  Thus, $\fif(\biud j\kappa r)=1_{L_j}$, as \eqref{eq:kzslgVczn} requires. 
Next, assume that $(j,\kappa)\neq (i,4)$. If $\kappa\neq 4$, then  \eqref{eq:vtVnmNtfgTRnsjbh} and \eqref{eq:GttKdKrcZRspRn} give that   $\fer_4(\biud j\kappa r)=0_{L_j}$, implying that $\fif(\biud j\kappa r)=0_{L_j}$, as required.
If $j\neq i$, then \eqref{eq:gThkFtmgMtBlsK} implies that all the joinands in \eqref{eq:bmTnkntJstRhmstrgtkr}
turn into $0_{L_j}$ when $\biud j\kappa r$ is substituted for $\vxi$, whereby $\gif(\biud j\kappa r)=0_{L_j}$ and so $\fif(\biud j\kappa r)=0_{L_j}$ again, as required. Now that we have proved  \eqref{eq:kzslgVczn}, the proof of Theorem \ref{thm:prd} is complete.
\end{proof}

\begin{proof}[Proof of Remark \ref{rem:ncdGnnJhNkszN}]
It suffices to exclude that $F=\mathbb Q(u)$ for some $u\in F$. Suppose the contrary and pick such a 
$u$. Then $u$ is transcendental and $\sqrt[80]{80}=f(u)/g(u)$ for some polynomials $f\in \mathbb Q[x]$ and $g\in \mathbb Q[x]\setminus\set 0$. 
Since $u$ is a root of the polynomial $f(x)^{80}-80 g(x)^{80}\in\mathbb Q[x]$, this polynomial is $0$. Hence, with  a $q\in\mathbb Q$ such that $g(q)\neq 0$,  $80=(f(q)/g(q))^{80}$. Thus  $\sqrt[80]{80}=f(q)/g(q)\in\mathbb Q$, which is a contradiction, as required. 
\end{proof}

\begin{proof}[Proof of Example \ref{expl:mZvsDnStSg}]
By the well-known multiplicativity of degrees and the primitive element theorem, see for example  Milne \cite[Proposition 1.20 and Theorem 5.1]{milne}, $F$ in Part (a) is $t=1$-generated. Hence,  Part (a) follows 
from \eqref{eq:nKmZfRtdrCh} and  \eqref{eq:fDnKnCspJhb}.
As the elements $\beta_i$ are independent, $t:=\mng F$ in Part (b) equals 80 by the fundamental theorem on transcendence bases; see for example Theorem 9.5 in Milne \cite{milne}. Therefore, \eqref{eq:nKmZfRtdrCh} and  \eqref{eq:fDnKnCspJha} imply Part (b). 
\eqref{eq:nKmZfRtdrCh} and \eqref{eq:fDnKnCspJhb} imply  Part (c).
To verify Part (d) for $|F|=19$, note that $k=10^{2046}$ is smaller than $\mu$ in \eqref{eq:ckHlrgjCnpzgN} by 
Table \ref{table:szbklngy}. If $F=\mathbb Q$, then $k\leq\mu=\aleph_0$ is trivial. Hence, $L^k$ is $5$-generated by \eqref{eq:fDnKnCspJhb}.
Since $\mng{\mathbb A}=\aleph_0$,   \eqref{eq:fDnKnCspJha} implies Part (e).
Finally, even without Remark \ref{rem:ncdGnnJhNkszN}, Part (f) follows from \eqref{eq:nKmZfRtdrCh} and \eqref{eq:fDnKnCspJhb} since $t=\mng F\in\set{1,2}$.
\end{proof}

\section{Appendix: Extracting Gelfand and Ponomarev's result from Z\'adori's proof}\label{sect:append-ZLprfn}

A lot in this paper depends on Gelfand and Ponomarev's theorem:

\begin{theorem}[Gelfand and Ponomarev \cite{gelfand}]\label{thm:gelfand}
If $3\leq n\in\Nplu$, $K$ is a prime field, and  $V=K^n$ is the $n$-dimensional vector space over $K$, then the subspace lattice $L(K^n):=\Sub V$ has a $4$-element generating set.
\end{theorem}

At the time of writing, the  \emph{old} website  \href{http://www.acta.hu/}{http://www.acta.hu/} of Acta Sci.\ Math.\ (Szeged) provides free access to  Z\'adori's paper \cite{zadori2}, while  Gelfand and Ponomarev's proof seems to be less available. Thus, we recall Z\'adori's construction briefly and point out how it proves  Theorem \ref{thm:gelfand}.  

Given a prime field $K$, an expression like $\vspan{-x,x,0,0,-2y,z,x+y}$  stands for the subspace $\set{(-x,x,0,0,-2y,z,x+y)\in K^7: x,y,z\in K}$. The subscript ``vs'' (from "vector space") distinguishes this subspace from the projective point $[-x,x,0,0,-2y,z,x+y]$ in the projective space $\psp{6}(K)$. Letting $c:=1$ in his paper \cite{zadori2}, Z\'adori's five subspaces turn into the following four subspaces.

\begin{definition}[{Z\'adori's subspaces \cite[for $c=1$]{zadori2}}]
For $n=2k+1\geq 3$, let
\begin{align*}
t_1=\nuid t n_1&:=\vspan{0,\dots,0,x_{k+1},\dots,x_{2k+1}},\cr
t_2=\nuid t n_2&:=\vspan{x_{1},\dots,x_{k},0,\dots,0}, \cr
t_3=\nuid t n_3&:=\vspan{x_{1},\dots,x_{k},0,x_{1},\dots,x_{k}},\text{ and} \cr
t_4=\nuid t n_4=t_5=\nuid t n_5&:=\vspan{x_{1},\dots,x_{k},x_{1},\dots,x_{k},0}.
\end{align*}
Furthermore, for $n=2k\geq 4$, let
\begin{align*}
t_1=\nuid t n_1&:=\vspan{0,\dots,0,x_{k+1},\dots,x_{2k}},\cr
t_2=\nuid t n_2&:=\vspan{x_{1},\dots,x_{k},0,\dots,0},\cr
t_3=\nuid t n_3&:=\vspan{x_{1},\dots,x_{k},x_{1},\dots,x_{k}},,\text{ and}\cr
t_4=\nuid t n_4=t_5=\nuid t n_5&:=\vspan{0,x_{2},\dots,x_{k},x_{2},\dots,x_{k},0}.
\end{align*}
\end{definition}

\begin{proof}[Proof Theorem \ref{thm:gelfand}, which is the  particular $c=1$ case of  Z\'adori \cite{zadori2}] 
Let  $\nuid T n:=\{\nuid t n_1$, \dots, $\nuid t n_4\}$, and let $\latgen {\nuid T n}$ stand for the  sublattice of $L(K^n)$ generated by $\nuid T n$. It suffices to show that $\latgen {\nuid T n} = L(K^n)$. 
For $n=3$,  $\set{\nuid t n_1,\dots,\nuid t n_4}$ generates 
$L:=\Subp K{K^n} \cong \Sub{\psp 2(K)}$ by Lemma \ref{lemma:mGszRzbTr} and Figure \ref{figipplndmnd}. 
The same holds for all $3\leq n\in\Nplu$, because 
the induction step from $\set{n-2,n-1}$ to $n$ is the same as in Z\'adori \cite{zadori2}, provided that we keep $c=1$  and 
 let $t_8:=t_7$ and $t_{12}:=t_{11}$ there.  This is how  Z\'adori \cite{zadori2} proves Theorem \ref{thm:gelfand}. 
 
The proof is ready at this point. However, for  the reader's convenience
and also because we can benefit from Section \ref{sect:coord},  we give more details. Note that although \eqref{eq:vSbWstcvlCtHlCrwld} is not in  Z\'adori's paper, this  does not mean a significant difference from his argument.

In what follows, using the case $n=3$  as the base of induction, we present the induction step. Actually, we present two sorts of induction steps; one for $n$  even and one for $n$ odd. But first, we formulate and prove an auxiliary statement; see \eqref{eq:vSbWstcvlCtHlCrwld} a bit later. 
For $u\in L(K^n)$, $\idl u$ will denote the \emph{principal ideal} 
$\set{v\in L(K^n): v\leq u}$. We claim that if we consider the following hyperplane 
\begin{equation*}H_i:=\vspan{x_1,\dots,x_{i-1},0,x_{i+1},\dots,x_n}
\label{eq:mHrHcsBjrTknRr}
\end{equation*}
and $G$ is a subspace of $K^n$ such that $G\nsubseteq H_i$ and $\dim G\geq 2$, then 
\begin{equation}
\idl{G}\cup\idl{H_i}\text{ generates }  L(K^n).
\label{eq:vSbWstcvlCtHlCrwld}
\end{equation}
It suffices to prove  \eqref{eq:vSbWstcvlCtHlCrwld} only
in the case when $i=n$ and $\dim G=2$. In this case, after passing from $V$ to the projective space $\psp {n-1}$ over $K$, $G$ is a line of $\psp{n-1}$ and $H_{n}$ is a hyperplane with codimension $1$, that is, $H_{n}$ is a coatom in $\Sub{\psp{n-1}}$. We treat $H_n$ as the hyperplane at infinity. Let $S$ stand for the sublattice generated by $\idl G\cup\idl{H_n}$ in $\Sub{\psp{n-1}}$. Let $p\in \psp {n-1}$ be an arbitrary point, that is, $\set p$ (which we denote by $p$ according to the conventions of the paper) is an atom of $\Sub{\psp{n-1}}$. We are going to show that $p\in S$. We can assume that $p\notin \idl G\cup\idl{H_n}$ in $\Sub{\psp{n-1}}$, as otherwise $p\in S$ is obvious. In particular, $p\notin H_n$ (understood geometrically), that is, $p$ is a finite point.
We know (say, from  Gr\"atzer \cite[page 376]{ggfound}) that whenever a subspace contains two distinct points of a line, then it contains all points of the line in question. We also know that each line has at least three points.
Hence, it follows from $G\nsubseteq H_n$ and \eqref{eq:mtzTbnkStmlDsl} that $G\setminus H_n$ contains two distinct points, $q_1$ and $q_2$.  Since $p\notin G=\prline {q_1}{q_2}$, we have that $\prline{p}{q_1}\neq \prline{p}{q_2}$, and so $p=\prline{p}{q_1}\wedge \prline{p}{q_2}$ in $\Sub{\psp{n-1}}$.
For $i\in[2]$, let $r_i$ denote the point at infinity on the line $\prline {p}{q_i}$; $r_i$ exists since each line has at least one point at infinity, it is in the hyperplane $H_n$, and it is uniquely determined since the finite points $p,q_i$ on $\prline {p}{q_i}$  exclude that $\prline {p}{q_i}\subseteq H_n$. 
As $p$, $q_i$, and $r_i$ are three distinct points on the same line, we have that 
\begin{equation*}
p=\prline{p}{q_1}\wedge \prline{p}{q_2}=\prline{r_1}{q_1}\wedge \prline{r_2}{q_2}=(r_1\vee q_1)\wedge (r_2\vee q_2)\in S,
\end{equation*}
proving the validity of \eqref{eq:vSbWstcvlCtHlCrwld}.

Next, to perform the induction step from $n-1$ (and, for an odd $n$, also from $n-2$) to $n$,  first we deal with the case when  $4\leq n=2k$ is even. Then we define\footnote{There will be no $t_5$ in this paper and there will be other gaps in the set of subscripts later. This makes it easier to see that the subspaces defined here are exactly  the ``$c:=1$ cases of the subspaces'' given in Z\'adori \cite{zadori2}, but now we do not need all of his subspaces.}
\begin{align*}
\nuid t n_6:=(\nuid t n_1\vee \nuid t n_4)\wedge \nuid t n_2=\vspan{0,x_2,\dots,x_k,0,\dots,0} \text{ and}\cr
\nuid t n_7:=(\nuid t n_1\vee \nuid t n_4)\wedge \nuid t n_3=\vspan{0,x_2,\dots,x_k,0,x_2,\dots,x_k}. 
\end{align*}
Let $B:=\set{\nuid t n_1,\nuid t n_6,\nuid t n_7,\nuid t n_4}$; it is a subset of $\latgen{\nuid T n}$. Since $B$ is the image of $\nuid T{n-1}$ under the ``natural\footnote{We use quotation marks around ``natural'' to indicate that not in a category theoretic sense.} isomorphism'' $ K^{n-1} \to  \vspan{0,x_2,\dots,x_n}=H_1$, 
the induction hypothesis implies that $G:=\idl {H_1}\subseteq \latgen{\nuid T n}$. Since $\nuid T n$ is invariant under the automorphism $K^n\to K^n$ defined by $(x_1,\dots,x_n)\mapsto (x_n,\dots,x_1)$, $\idl {H_n}\subseteq \latgen{\nuid T n}$ also holds. Hence, \eqref{eq:vSbWstcvlCtHlCrwld} implies that $\latgen{\nuid T n}=L(K^n)$, as required.

Second, we assume that $n=2k+1\geq 5$. Then $\latgen{\nuid T n}$ contains
\allowdisplaybreaks{
\begin{align*}
&\nuid t n_6:=(\nuid t n_2\vee \nuid t n_3)\wedge \nuid t n_1=\vspan{0,\dots,0,x_{k+2},\dots,x_{2k+1}},\cr
&\nuid t n_7:=(\nuid t n_2\vee \nuid t n_3)\wedge \nuid t n_4=\vspan{0,x_2,\dots,x_k,0,x_2,\dots,x_k,0}\cr
&\nuid t n_9:=(\nuid t n_2\vee \nuid t n_4)\wedge \nuid t n_1=\vspan{0,\dots,0,x_{k+1}\dots,x_{2k},0}\cr
&\nuid t n_{10}:=(\nuid t n_2\vee \nuid t n_4)\wedge \nuid t n_3=\vspan{x_1,\dots,x_{k-1},0,0,x_1,\dots,x_{k-1},0}\cr
&\nuid t n_{11}:=(\nuid t n_9\vee \nuid t n_{10})\wedge \nuid t n_4=\vspan{x_1,\dots,x_{k-1},0,x_1,\dots,x_{k-1},0,0}\cr
&\nuid t n_{13}:=(\nuid t n_9\vee \nuid t n_{10})\wedge \nuid t n_2=\vspan{x_1,\dots,x_{k-1},0,\dots,0}\cr
\end{align*}}
Since $\set{\nuid t n_6, \nuid t n_2, \nuid t n_3, \nuid t n_7}$  corresponds to $\nuid T{n-1}$ under the ``natural isomorphism'' $K^{n-1}\to H_{k+1}$, 
the induction hypothesis gives that $\idl{H_{k+1}}\subseteq \latgen{\nuid T n}$. As $\set{\nuid t n_9, \nuid t n_{13},  \nuid t n_{10},  \nuid t n_{11}}$ corresponds to $\nuid T{n-2}$ under the ``natural isomorphism'' $K^{n-2}$ $\to$ $H_k\cap H_{n}:=G$, the induction hypothesis yields also that $\idl{G} \subseteq \latgen{\nuid T n}$. 
Since $G\nsubseteq H_{k+1}$ and $\dim(G)=n-2\geq 3\geq 2$,
we can use \eqref{eq:vSbWstcvlCtHlCrwld} (with $i:=k+1$) to conclude that  $\latgen{\nuid T n}=L(K^n)$, completing the induction step and the proof of Theorem \ref{thm:gelfand}.
\end{proof}

\section{Appendix: the Maple program mentioned in Footnote \ref{foOt:mapLe}}\label{sect:maple}

This section presents two Maple programs. 

\subsection*{The following short program computed the data Table \ref{table:szbklngy}}

{\small
\begin{verbatim}
> restart; #with(combinat):
> gbc:=proc(q,m,r) local i,j,k,sz,nev,thisisit;
>   sz:=1; nev:=1;
>  for i from m-r+1 to m do  sz:=sz*(1-q^i) od;
>  for i from 1 to r do nev:=nev*(1-q^i) od;
>  thisisit:=round(evalf(sz/nev));
> end:
> for q from 2 to 19 do
>  if member(q, {2,3,4,5,7,8,9,11,13,16,17,19})
>  then d:=80: ehat:=gbc(q,d,floor(d/2)): print(" "): 
>   print(cat("d=",d,", q=",q," d chooses d/2 w.r.t. q=",
>    ehat,", and its log[10]=", evalf(log[10](ehat)) )):
>  fi:
> od:
\end{verbatim}
}

\subsection*{
We continue with the Maple program mentioned in Footnote \ref{foOt:mapLe}}

{\small
\begin{verbatim}
> restart; #Computation in the Fano plane
> # The program contains some parts, called "tests". Running
> # these parts can increase your trust in the program.
> # To run these parts, delete the hash marks (#) from them.
> 
> #   PART 1: ENTERING THE DESCRIPTION OF THE FANO PLANE
> 
> pnam:=array(1..7): #The names of the points in the paper
>  pnam[1]:="a1": pnam[2]:="a2": pnam[3]:="a3": 
>  pnam[4]:="b1": pnam[5]:="b2": pnam[6]:="b3": 
>  pnam[7]:="c":
> lnam:=array(8..14): #Lines names in the paper;
>  lnam[8]:="u1": lnam[9]:="u2": 
>  lnam[10]:="u3": lnam[11]:="v1":
>  lnam[12]:="v2": lnam[13]:="v3": lnam[14]:="w":  
> line:=array(8..14): #The lines in the paper
>  line[7+1]:={2,3,3+1}: line[7+2]:={1,3,3+2}:
>  line[7+3]:={1,2,3+3}: line[7+3+1]:={1,3+1,7}:
> line[7+3+2]:={2,3+2,7}:line[7+3+3]:={3,3+3,7}: 
>  line[7+7]:={3+1,3+2,3+3}: #Each line is a set of points;
>  #the program treats the points numbers while computing
>  #but uses their names, stored in pnam, when printing.
> L:=array(0..15): #The subspace lattice of the Fano plane
> for i from 1 to 7 do L[i]:={i} od:# 
> for i from 8 to 14 do L[i]:=line[i] od: L[0]:={}: L[15]:={}:
> for i from 1 to 7 do L[15]:=L[15] union {i} od:#
> lnotat:=array(0..15): 
>  #The notations of the subspaces in the paper
>  #like "0", "a1", "u2", or "1".
>  lnotat[0]:="0":lnotat[15]:="1":
>  for i from 1 to 7 do lnotat[i]:=pnam[i] od:
>  for i from 8 to 14 do lnotat[i]:=lnam[i] od:
> leq:=proc(x,y) local r; #Describing the order
>   if x=x intersect y then r:=1 else r:=0 fi
> end:#
> SetToName:=proc(x) local i,r; #Name: what the paper uses
>  #E.g., SetToName=({2,4,3}) = "u1"
>  r:="Non-recognizable":
>  for i from 0 to 15 do if x=L[i] then r:=lnotat[i] fi
>  od:  r:=r:     
> end: #End of SetToName
> SetToStr:=proc(x) 
>  #E.g., SetToStr({2,3,4})="{a2,a3,b1}"  
>  local i,r,needscomma;
>  r:="{": needscomma:=0:
>  for i from 1 to 7 do
>   if leq(L[i],x)=1 then 
>     if needscomma=1 then r:=cat(r,",",lnotat[i])
>     else r:=cat(r,lnotat[i]): needscomma:=1
>     fi:
>   fi:
>  od: #end of "for i" loop
>  r:=cat(r,"}"):
> end: #End of procedure SetToStr 
> 
> #           PART 2: LISTING THE DETAILS OF THE FANO PLANE
> 
> lstr:=array(0..15):#The subspaces in string forms
>  # like  "u1={a2,a2,b1}", "a1={a1}",  or "0={}"
>  for i from 0 to 15 do 
>   lstr[i]:=cat(lnotat[i],"=",SetToStr(L[i])) 
>  od:  #end of the "for i" loop
print("The details of the subspace lattice L"):
print("  of the Fano plane are as follows:"):
for i from 0 to 15 do print(cat(lstr[i],
                  " (stored in L(",i,")")) od:
               "The details of the subspace lattice L"
                "  of the Fano plane are as follows:"
                        "0={} (stored in L(0)"
                      "a1={a1} (stored in L(1)"
                      "a2={a2} (stored in L(2)"
                      "a3={a3} (stored in L(3)"
                      "b1={b1} (stored in L(4)"
                      "b2={b2} (stored in L(5)"
                      "b3={b3} (stored in L(6)"
                       "c={c} (stored in L(7)"
                   "u1={a2,a3,b1} (stored in L(8)"
                   "u2={a1,a3,b2} (stored in L(9)"
                   "u3={a1,a2,b3} (stored in L(10)"
                   "v1={a1,b1,c} (stored in L(11)"
                   "v2={a2,b2,c} (stored in L(12)"
                   "v3={a3,b3,c} (stored in L(13)"
                   "w={b1,b2,b3} (stored in L(14)"
              "1={a1,a2,a3,b1,b2,b3,c} (stored in L(15)"
> #
> #               PART 3: COMPUTING THE JOIN IN L
> #   
> which:=proc(x) local i,r; # x is subspace
>  r:=-1; for i from 0 to 15 do if x=L[i] then r:=i fi od;
>  # if r=-1 then print(" !!! -1 means: NOT IN L !!!"): fi:
>  r:=r;
> end: #And now a few tests with "which":
> #The built-in operation "intersect" is good for meet.
> join:=proc(x,y) local z,i,r:
>  z:=L[15]: #The top element
>  for i from 0 to 14 do
>   if (leq(x,L[i])=1) and (leq(y,L[i])=1) 
>      then z:=z intersect L[i]
>   fi
>  od: #End of the "for i" loop
>  r:=z:
> end: #End of procedure join 
> 
> # Test:  in the next two lines, we test some joins in L:
> #a:={1}:b:={2,4,3}: c:=join(a,b); print(cat
> #(SetToName(a)," join ",SetToName(b),"=",SetToName(c))):
> 
> 
> #       PART 4: SEARCH IN S
> 
> S:=array(1..257,1..2):#The sublattice to be generated   
> Ssize:=0: #At present, S is the emptyset
> whereInS:=proc(x,y) local r,i: 
>                        #Finds an element of L^2 in S
>  r:=-1:
>  for i from 1 to Ssize do
>   if (x=S[i,1]) and (y=S[i,2]) then r:=i
>   fi:
>  od: r:=r:
> end: #End of procedure whereInS; it will be tested later.
> 
> #       PART 5: COMPUTING WHAT S GENERATES
> 
> generating:=proc()  local i,j,z1,z2,m1,m2,found,oldSize; 
>     global S,Ssize;  
>  #Computes what (S[1,1],S[1,2]), ... , 
>  # (S[Ssize,1],S[Size,2]) generates,  puts it into S,
>  # and increases Ssize  
>  found:=true:
>  while found=true
>  do found:=false: oldSize:=Ssize:
>   for i from 1 to oldSize-1
>   do for j from i+1 to oldSize
>    do z1:=join(S[i,1],S[j,1]): z2:=join(S[i,2],S[j,2]):
>       m1:=S[i,1] intersect S[j,1]: 
>       m2:=S[i,2] intersect S[j,2]: 
>       if whereInS(z1,z2)=-1 then
>        found:=true: Ssize:=Ssize+1: 
>        S[Ssize,1]:=z1: S[Ssize,2]:=z2:
>       fi: # New join added 
>       if whereInS(m1,m2)=-1 then
>        found:=true: Ssize:=Ssize+1: 
>        S[Ssize,1]:=m1: S[Ssize,2]:=m2:
>       fi: # New meet added 
>    od: # for j
>   od: # for i 
>  od: #while found; now S is the sublattice generated.
> end: #End of procedure generating;
> #it will be tested later, after initialization
> 
> #      PART 6: CONVERTING A ROW OF S TO TEXT
> 
> Sname:=proc(i) local i1,i2,r:#E.g, Sname(1)="(a1,a1)"
>  i1:=which(S[i,1]); i2:=which(S[i,2]);  
>  if (i1=-1) or (i2=-1)
>  then print("Something is wrong here"): r:=""
>  else  r:=cat("(",lnotat[i1],",",lnotat[i2],")")
>  fi: r:=r:  
> end: #End of proceture Sname, to be tested later. 
> #
> #Test:    FIRST TEST (OPTIONAL)
> 
> #Testing what 3 points on a line and a further point
> #generate; and testing Sname and whereInS, too.   
> #for i from 1 to 4 do S[i,1]:=L[i]: S[i,2]:=L[i]:
> #od:  Ssize:=4: print(cat("The subset of L^2:")): 
> #for i from 1 to Ssize do print(Sname(i)) od:  
> #generating(): 
> #print(cat("generates the following ",
> #            Ssize,"-element sublattice:"));   
> #for i from 1 to Ssize do print(Sname(i)) od:
> #print("(A whereInS-test:"): 
> #print(cat(Sname(L[8]),"=",L[8]," it is the ",
> # whereInS(L[8],L[8]),"-th" )): 
> 
> #Test:    SECOND TEST (OPTIONAL)
> #Testing what 4 points in general position  generate    
> #for i from 1 to 3 do S[i,1]:=L[i]: S[i,2]:=L[i]:
> #od:  S[4,1]:=L[7]: S[4,2]:=L[7]: Ssize:=4: generating(): 
> #print(cat("The following ",Ssize,
> #                   "-element sublattice is generated",
> #      " by its first four elements:"));   
> #for i from 1 to Ssize do print(Sname(i)) od: 
> 
> #      PART 7: THE MAIN COMPUTATION
> #    
> for i from 1 to 3 do S[i,1]:=L[i]: S[i,2]:=L[i]:
> od:  S[4,1]:=L[7]: S[4,2]:=L[14]: Ssize:=4: 
> print("The following 4 elements of L^2:"):
> txt:=Sname(1):for i from 2 to Ssize do 
>             txt:=cat(txt,", ", Sname(i)) od: print(txt):  
> generating():print("generate a ",
> Ssize,"-element sublattice,"):   
> print("which consists of the following elements:"):  
for i from 1 by 5 to Ssize do  txt:="":  
 for j from 0 to 4 do
  if i+j<Ssize  then txt:=cat(txt,Sname(i+j),", "):
  fi:
  if i+j=Ssize then txt:=cat(txt,Sname(i+j),"."):  
  fi
 od: print(txt): 
od:
a1:=lnotat[15]: a2:=lnotat[0]:
print(cat("The position of (", a1, ",", a2, 
                  ") is ",whereInS(L[15],L[0]))):
print("(-1  means that not found)"):   
                  "The following 4 elements of L^2:"
                  "(a1,a1), (a2,a2), (a3,a3), (c,w)"
              "generate a ", 50, "-element sublattice,"
            "which consists of the following elements:"
            "(a1,a1), (a2,a2), (a3,a3), (c,w), (u3,u3), "
             "(0,0), (u2,u2), (v1,1), (u1,u1), (v2,1), "
              "(v3,1), (1,1), (0,a1), (0,a2), (0,a3), "
             "(0,b3), (0,b2), (0,b1), (a1,u3), (a2,u3), "
           "(b3,u3), (a1,u2), (b2,u2), (a3,u2), (b1,u1), "
             "(c,1), (a2,u1), (a3,u1), (a1,v1), (u3,1), "
             "(u2,1), (a2,v2), (u1,1), (a3,v3), (0,u3), "
              "(0,u2), (0,u1), (0,v1), (b1,1), (a2,1), "
              "(a3,1), (0,v2), (a1,1), (b2,1), (0,v3), "
                "(b3,1), (0,w), (w,1), (0,1), (0,c)."
                    "The position of (1,0) is -1"
                     "(-1  means that not found)"

\end{verbatim}
}
This program (called ``worksheet'' in Maple) is also available from the \href{http://www.math.u-szeged.hu/~czedli}{author's website}.


\begin{thebibliography}{99}


\bibitem{artmann}
Artmann, B.: On coordinates in modular lattices with a homogeneous basis. Illinois J. Math. \tbf{12}, 626--648 (1968)


\bibitem{czgQuogen}
Cz\'edli, G.:
Four-generated quasiorder lattices and their atoms in a four generated sublattice. 
Communications in Algebra \tbf{45} 4037--4049 (2017)

\bibitem{czgboolegen}
Cz\'edli, G.:  Generating Boolean lattices by few elements and exchanging session keys.
\href{http://arxiv.org/abs/2303.10790}{arXiv:2303.10790}


\bibitem{czgskublics}
Cz\'edli, G.,  Skublics, B.:
The ring of an outer von Neumann frame in modular lattices.
Algebra Universalis \tbf{64} 187--202 (2010)


\bibitem{daypick}
Day, A., Pickering, D.: The coordinatization of Arguesian lattices.
Trans. Amer. Math. Soc. \tbf{278}, 507--522 (1983)

\bibitem{freese}
Freese, R: The variety of modular lattices is not generated by its finite members.
Trans. Amer. Math. Soc. \tbf{255}, 277--300 (1979)


\bibitem{gelfand}
Gelfand, I.M., Ponomarev, V.A.:
Problems of linear algebra and classification of quadruples of subspaces in a finite dimensional vector space. 
Hilbert Space Operators, Coll. Math. Soc. J. Bolyai 5, Tihany, Hungary (1970)
%

\bibitem{ggfound}
Gr\"atzer, G.: Lattice Theory: Foundation. 
Birkh\"{a}user,  Basel (2011)

\bibitem{herrmann}
Herrmann, C.: On the equational theory of modular lattices.
Proc. Univ. of Houston Lattice Theory Conference, Houston, 105--118 (1973)

\bibitem{herrmann84}
Herrmann, C.: On the arithmetic of projective coordinate systems. Trans. Amer. Math. Soc. \tbf{284}, 759--785 (1984)


\bibitem{herrmannhuhn}
Herrmann, C., Huhn, A.P.:
Lattices of normal subgroups which are generated by frames.
Colloq. Math. Soc. J. Bolyai 14. Lattice Theory, Szeged (Hungary),  97--136 (1974)


\bibitem{huhn}
Huhn, A.P.: Schwach distributive Verb\"ande I. 
Acta Sci. Math. (Szeged) \tbf{33}, 297--305 (1972) (in German)


\bibitem{milne}
Milne, J.S.: Fields and Galois Theory.
Kea Books\footnote{See also \href{https://www.jmilne.org/math/CourseNotes/FT.pdf}{https://www.jmilne.org/math/CourseNotes/FT.pdf} for an earlier version.}, Ann Arbor (2022) 

\bibitem{neumann}
von Neumann, J.: Algebraic theory of continuous geometries. 
Proc. Nat. Acad. Sci. U.S.A. \tbf{23}, 16--22 (1937)

\bibitem{bookvonneumann}
von Neumann, J.: Continuous Geometry. (Foreword by Israel Halperin), Princeton University Press, Princeton (1960)


\bibitem{ohara}
O'Hara, K.M.: Unimodality of Gaussian Coefficients:
A Constructive Proof. 
Journal of Combinatorial Theory A \tbf{53}, 29--52 (1990)


\bibitem{strietz}
Strietz, H.: \"Uber Erzeugendenmengen endlicher Partitionverb\"ande. Studia Sci. Math. Hungarica \tbf{12}, 1--17 (1977) (in German) 

\bibitem{vanstone}
Vanstone, S.A., van Oorschot, P.C.:
An Introduction to Error Correcting Codes with Applications.
Kluwer, Boston--Dordrecht--London (1989)


\bibitem{veblenyoung}
Veblen, O., Young, J.W.:
Projective Geometry I.
Ginn and Co., Boston (1910) 

\bibitem{wild}
Wild, M.: Cover-preserving embedding of modular lattices into partition lattices.
Discrete Mathematics \tbf{112}, 207--244 (1993)

\bibitem{willefm}
Wille, R.: On free modular lattices generated by finite chains.
Algebra Universalis \tbf{3}, 131--138 (1973)


\bibitem{zadori2}
Z\'adori, L.: Subspace lattices of finite vector spaces are 5-generated.
Acta Sci. Math. (Szeged) 74 (2008), 493--499.

\end{thebibliography}
\end{document}